\documentclass[10pt]{amsart}
\usepackage{amsthm, amsfonts, amssymb, amsmath, amscd, enumitem}
\usepackage[all]{xy}
\usepackage{lmodern}
\usepackage[margin=3.6cm]{geometry}

\newcommand{\Too}{\longrightarrow}

\newcommand{\onto}{\twoheadrightarrow}

\newcommand{\act}{\curvearrowright}

\newcommand{\C}{\mathbb{C}}
\newcommand{\R}{\mathbb{R}}
\newcommand{\Q}{\mathbb{Q}}
\newcommand{\Z}{\mathbb{Z}}

\newcommand{\M}{\mathrm{M}}

\newcommand{\G}{\Gamma}

\newcommand{\Cred}{C^*_r}
\newcommand{\B}{\mathcal{B}}
\newcommand{\Comp}{\mathcal{K}}
\newcommand{\LL}{\mathcal{L}}
\newcommand{\hi}{C(\bd \G)\rtimes \G}

\newcommand{\hiX}{C(\bd X)\rtimes \G}

\newcommand{\rcross}{\rtimes_\mathrm{r}}
\newcommand{\ralg}{\rtimes_\mathrm{alg}}

\newcommand{\K}{\mathrm{K}}
\newcommand{\KK}{\mathrm{KK}}
\newcommand{\EG}{\mathcal{E}\G}

\newcommand{\pnt}{\textup{pnt}}

\newcommand{\Eul}{\mathrm{Eul}}

\newcommand{\RKK}{\textup{RKK}}

\newcommand{\Calk}{\B(\ell^2\G)/\Comp(\ell^2\G)}
\newcommand{\tauop}{\tau^{\textup{op}}}
\newcommand{\lambdaop}{\lambda^{\textup{op}}}

\newcommand{\ind}{\mathrm{index}}

\newcommand{\E}{\mathrm{E}\hspace{.07pc}}
\newcommand{\Ebar}{\overline{\mathrm{E}}\hspace{.07pc}}

\newcommand{\dev}{\sigma}

\newcommand{\sop}{s^{\mathrm{op}}}

\newcommand{\Hyp}{\mathbb{H}} %hyperbolic plane

\newcommand{\lambdamuop}{\lambda_\mu^{\mathrm{op}}}
\newcommand{\projcon}{P_{\ell^2\G}}

\newcommand{\e}{\epsilon}

\newcommand{\la}{\langle}
\newcommand{\ra}{\rangle}
\newcommand{\bd}{\partial}
\newcommand{\Lip}{\mathrm{Lip}}

\newcommand{\hdim}{\mathrm{hdim}}
\newcommand{\visdim}{\mathrm{visdim}\:}
\newcommand{\topdim}{\mathrm{topdim}\:}
\numberwithin{equation}{section}
\newcommand{\SO}{\mathrm{SO}}
\newcommand{\SU}{\mathrm{SU}}
\newcommand{\Sp}{\mathrm{Sp}}

\newtheorem{thm}{Theorem}[section]
\newtheorem{lem}[thm]{Lemma}
\newtheorem{cor}[thm]{Corollary}
\newtheorem{prop}[thm]{Proposition}
\newtheorem*{thma}{Theorem A}
\newtheorem*{thmb}{Theorem B}

\theoremstyle{definition}
\newtheorem{fact}[thm]{Fact}
\newtheorem{defn}[thm]{Definition}
\newtheorem{rem}[thm]{Remark}
\newtheorem{ex}[thm]{Example}

\newtheorem*{notation}{Notation}

\begin{document}
\title{K-homological finiteness and hyperbolic groups}
\author{Heath Emerson}
\address[H.E.]{Department of Mathematics and Statistics, University of Victoria, Victoria (Canada)}
\email{hemerson@uvic.ca}
\author{Bogdan Nica}
\address[B.N.]{Mathematisches Institut, Georg-August Universit\"at G\"ottingen, G\"ottingen (Germany)}
\email{bogdan.nica@gmail.com}
\date{\today}

\begin{abstract}
Motivated by classical facts concerning closed manifolds, we introduce a strong finiteness property in K-homology. We say that a C*-algebra has \emph{uniformly summable K-homology} if all its K-homology classes can be represented by Fredholm modules which are finitely summable over the same dense subalgebra, and with the same degree of summability. We show that two types of C*-algebras associated to hyperbolic groups -- the C*-crossed product for the boundary action, and the reduced group C*-algebra -- have uniformly summable K-homology. We provide explicit summability degrees, as well as explicit finitely summable representatives for the K-homology classes. 
\end{abstract}

\maketitle

\section*{Introduction}
A \emph{Fredholm module} over a unital C*-algebra \(A\) consists of a 
representation \(\pi \colon A \to \B(H)\) of \(A\) on a Hilbert space, and an operator 
$T\in \B(H)$, which in the even case is an essential unitary and in the odd case an 
essential projection, and which in both cases is required to essentially commute with the representation: $[\pi (a), T] \in \Comp(H)$ for all $a\in A$. The definition of Fredholm module was motivated by elliptic operator theory: if \(M\) is a smooth, compact manifold, then any zero-order elliptic pseudodifferential operator \(T\) on sections \(L^2(E)\) of a bundle over \(M\), determines a Fredholm module with the obvious representation of $C(M)$ on $L^2(E)$ by multiplication. Since elliptic operators induce maps on K-theory by an index theoretic construction, this led Atiyah, and subsequently Kasparov, to describe the K-homology of a C*-algebra $A$ as equivalence classes of Fredholm modules over \(A\). 

Connes' early work on cyclic cohomology, the noncommutative analogue of de Rham theory, and on the noncommutative Chern character, a map from K-homology to cyclic theory, suggested the importance of the finite summability 
condition on a Fredholm module that 
\begin{align*}
[\pi (a), T] \in \LL^p(H)\; \textup{for dense \(a\in A\)}
\end{align*}
where \(\LL^p(H)\) is the Schatten ideal of $p$-summable compact operators. Connes showed how to associate a canonical cyclic cocycle, the so-called \emph{character}, to a finitely summable Fredholm module, and how to use the character for computing the index pairing between the K-theory of 
 \(A\) and the K-homology class of the Fredholm module. This is just one aspect of a larger landscape, that of \emph{quantized calculus} \cite[Ch.IV]{Connes}, depending on finite summability.
  
Examples of finitely summable Fredholm modules over C*-algebras are thus of considerable interest in noncommutative geometry, and by this stage there have been numerous constructions of them, but one can state a \emph{theorem} about their existence in the classical situation, using elliptic operator theory and Poincar\'e duality. If \(M\) is a closed manifold, then \(M\) has Poincar\'e duality with its tangent bundle. This implies that every K-homology class for \(M\) is represented by a pseudodifferential operator of order zero. Classical spectral estimates for pseudodifferential operators imply that the singular values of commutators \([f, T]\), where \(f\) is a smooth function and \(T\) is pseudodifferential, satisfy the asymptotic law $s_n=O(n^{-1/\dim M})$. It follows from this that every K-homology class for \(M\) is represented by a Fredholm module which is \(p\)-summable over \(C^\infty (M)\) for each \(p>\dim M\). The moral is that the K-homology of a closed manifold exhibits finiteness, in the sense that all K-homology classes can be represented by finitely summable Fredholm modules, in a strong form: both the smooth subalgebra and the degree of summability may be taken uniform across all K-homology classes. 

The main focus of the present article is on noncommutative manifestations of this strong finiteness phenomenon. We say that a C*-algebra has \emph{uniformly summable K-homology} if all its K-homology classes can be represented by Fredholm modules which are finitely summable over the same smooth subalgebra, and with the same degree of summability.

 \begin{thma}
 Let $\G$ be a regular, torsion-free hyperbolic group, and let $\bd \G$ denote its boundary. Then the crossed-product C*-algebra $\hi$ has uniformly summable K-homology.
 \end{thma} 

The summability degree is the Hausdorff dimension of the boundary, more precisely a suitable interpretation thereof, and it is obtained by analytic means from the $\G$-invariant H\"older structure of the boundary. The same structure is responsible for the smooth subalgebra. 
 
Fredholm modules for the crossed-product C*-algebra $\hi$ restrict to Fredholm modules for the reduced group C*-algebra $\Cred\G$, preserving summability. However, additional effort and new ingredients are needed in order to represent all the K-homology classes of $\Cred\G$. Our main result in this direction, Theorem~\ref{reps for reduced} below, is somewhat technical due to certain issues connected with the Baum - Connes conjecture. At this point, we merely quote the following concrete consequence.

 \begin{thmb}
 \label{thmb}
 Let $\G$ be one of the following:
\begin{itemize}[leftmargin=20pt, itemsep=3pt]
\item a finitely generated free group;
\item a torsion-free cocompact lattice in $\SO(n,1)$ or $\SU(n,1)$;
\item a torsion-free $C'(1/6)$ small-cancellation group with one more generator than relators. 
\end{itemize}
Then the reduced group C*-algebra $\Cred\G$ has uniformly summable K-homology over the group algebra $\C\G$.
 \end{thmb}

Hyperbolicity, in the sense of Gromov, is a coarse notion of negative curvature. A hyperbolic space is a geodesic metric space all of whose geodesic triangles are uniformly thin. A group is said to be hyperbolic if it admits a \emph{geometric} - that is, isometric, proper, and cocompact - action on a hyperbolic space. Fundamental examples of hyperbolic groups are finitely generated free groups, cocompact lattices in $\SO(n,1)$, $\SU(n,1)$, or $\mathrm{Sp}(n,1)$ (`classical hyperbolic groups'), as well as $C'(1/6)$ small-cancellation groups. Throughout this paper, hyperbolic groups are assumed to be non-elementary, meaning that we discard the virtually cyclic groups. 

The boundary $\bd \G$ of a hyperbolic group $\G$ is a compact Hausdorff space carrying a natural action of $\G$ by homeomorphisms. Our standing assumption that $\G$ is non-elementary translates into $\bd \G$ having uncountably many points. For instance, the boundary of a free group is a Cantor set, and the boundary of a classical hyperbolic group is a sphere. The \(\G\)-action on \(\bd \G\) is topologically amenable \cite{Ada}. This means that the C*-algebra $\hi$ is nuclear, and that the full and the reduced crossed products coincide \cite{Ana02}. The boundary compactification \(\overline{\G} = \G \cup \bd \G\) carries a $\G$-action as well, and it is a coarse compactification -- in the sense that a ball of uniform size in \(\G\) becomes small in the topology of the compact space \(\overline{\G}\) when translated out to the boundary. The action of \(\G\) on the boundary \(\bd \G\) is minimal and exhibits a north-south dynamics, making the C*-algebra crossed-product \(\hi\) simple and purely infinite \cite{Ana97, LS96}. If \(\G\) is a free group, then \(\hi\) is canonically a Cuntz - Krieger algebra with transition matrix simply coding that a generator cannot be followed by its inverse. If \(\G\) is a surface group, \(\G = \pi_1(M)\) where $M$ is a closed surface of genus at least $2$, then $\bd\G$ is the \(1\)-sphere $S^1$ and the groupoid \(\G\ltimes \bd \G\) is strongly Morita equivalent to the holonomy groupoid of the canonical foliation of \(\tilde{M} \times_\G S^1\) by the projections to \(\tilde{M}\times_\G S^1\) of copies of \(\tilde{M}\). 

The K-theory of the boundary crossed-product \(\hi\) has been investigated in \cite{Emerson:Euler}. The main 
result of this reference is a Gysin sequence for boundary actions which computes the map on K-theory induced by the inclusion \(i\colon C^*_r \G \to \hi\). Its K-homology version is described in this article, and it plays an important role in our results. Another K-theoretic feature of \(\hi\) is that it exhibits Poincar\'e self-duality: the K-theory and the K-homology of \(\hi\) are isomorphic with a parity shift  \cite{Emerson}. The isomorphism is implemented by a cup-cap product map $\Delta \cap \colon \K_*(\hi) \to \K^{*+1}(\hi)$ with a suitable class $\Delta \in \K^1((\hi)\otimes (\hi))$. For the purposes of this paper, the important feature of Poincar\'e self-duality is the surjectivity of $\Delta \cap$, which leads to an explicit description of the K-homology of $\hi$ as a function of its K-theory. The proof of Poincar\'e duality given in \cite{Emerson} requires $\G$ to be torsion-free, and its boundary $\bd\G$ to admit a continuous self-map without any fixed points. This technical condition on the boundary is the \emph{regularity} assumption in Theorem A. We do not know of any hyperbolic group which fails to be regular. 

Our rank-$1$ results in Theorem B are in stark contrast to the higher rank situation. As shown by Puschnigg \cite{Pusch}, no non-trivial K-homology class for the reduced C*-algebra of a higher rank lattice can be represented by a Fredholm module which is finitely summable over the group algebra. This very opposite behaviour with respect to K-homological finiteness is reminiscent of the following sharp distinction between rank-$1$ lattices and higher rank lattices, also involving a notion of finite summability. Every hyperbolic group admits a proper isometric action on an $L^p$-space for large enough $p>1$ \cite{Yu}, \cite{Nica}, but every isometric action of a higher rank lattice on an $L^p$-space, $p>1$, fixes a point \cite{BFGM}.

It should be noted that, despite our strong finiteness results at the level of Fredholm modules (i.e., \emph{bounded} K-cycles), neither the boundary crossed-product $\hi$ nor the reduced C*-algebra $C^*_r\G$ support any finitely summable spectral triples (i.e., \emph{unbounded} K-cycles). For $\hi$ this is due to the lack of a trace \cite[Thm.8]{Con89}, whereas for $C^*_r\G$ the reason is the non-amenability of $\G$ \cite[Thm.19]{Con89}, \cite[Thm.1 in IV.9.$\alpha$]{Connes}.

The present work expands, and supersedes, our preprint \cite{babyEN}.

\smallskip
\noindent\textbf{Other results.} Lott \cite{Lott} employed elliptic operator methods to construct K-homology cycles, some of which are finitely summable, for the crossed-product C*-algebra arising from the action of a subgroup of $\SO(n,1)$ on its limit set. Lott's constructions are not obviously related to ours. 

Rave \cite{Rave} proved that AF C*-algebras have, in the terminology of this paper, uniformly summable K-homology. Incidentally, \cite{Rave} also contains a good account of the fact that the commutative C*-algebra $C(M)$ of a closed manifold $M$ has uniformly summable K-homology.

Very recently, Goffeng and Mesland \cite{GoM} have addressed the issue of uniform summability for K-homology in the case of Cuntz - Krieger C*-algebras. This is a family which bears some analogies to the boundary C*-crossed products considered herein. In \cite{GoM} it is shown, among other things, that the odd K-homology of a Cuntz - Krieger C*-algebra is uniformly summable.

\smallskip
\noindent\textbf{Acknowledgements.} We thank Nigel Higson, Vadim Kaimanovich, Misha Kapovich, Bruce Kleiner, Georges Skandalis, and Bob Yuncken for correspondence or discussions. The first author acknowledges support from an NSERC Discovery grant. The second author thanks the Pacific Institute for the Mathematical Sciences, and the Alexander von Humboldt Foundation for their support. We also thank the last referee for a careful reading of the paper, and for constructive comments.

 %%%%%%%%%%%%%%%%%%%%%%%%%%%%%%%%%%%%%%%%%%%%%
\section{Outline} 
We now describe our results and our approach in more detail. Let us start with a conceptual clarification of what exactly is the boundary of a hyperbolic group. Having fixed a (non-elementary) hyperbolic group $\G$, by a \emph{geometric model} for $\G$ we mean a hyperbolic space on which $\G$ acts geometrically. Some groups come with ready-made geometric models, e.g., for a cocompact lattice in $\SO(n,1)$ the $n$-dimensional real hyperbolic space $\Hyp^n$ is such a space. Otherwise, a geometric model can be manufactured as the Cayley graph with respect to a finite generating set - for instance, free groups admit regular trees as geometric models. The important point is that boundaries of geometric models for $\G$ are $\G$-equivariantly homeomorphic. Thus, the boundary of $\G$ should be understood as the boundary of any geometric model for $\G$.

Let $X$ be a geometric model for $\G$, so $\bd X$ is a topological avatar of the boundary of $\G$. There is a natural collection of \emph{visual metrics} on $\bd X$, all assigning a finite Hausdorff dimension to $\bd X$. The \emph{visual dimension} of the boundary, denoted $\visdim \bd X$, is the infimal Hausdorff dimension over the family of visual metrics. In particular, the visual dimension is at least as large as the topological dimension. The Hausdorff measures defined by visual metrics are comparable, in the sense that they are within constant multiples of each other. This prompts us to define a  \emph{visual probability measure} on $\bd X$ as a Borel probability measure which is comparable with some (equivalently, each) Hausdorff measure defined by a visual metric. 

Consider now the crossed product $\hiX$, a C*-avatar of $\hi$. Endowing $\bd X$ with a visual probability measure $\mu$, we obtain a faithful representation of $C(\bd X)$ on $L^2(\bd X,\mu)$ by multiplication. This induces, in turn, a faithful representation $\lambda_\mu$, the \emph{left regular representation with respect to $\mu$}, of $\hiX$ on $\ell^2(\G,L^2(\bd X,\mu))$. We also let $\projcon$ be the projection onto $\ell^2\G$, regarded as constant functions on $\bd X$.

The basic idea for our construction of Fredholm modules, and the relationship 
to Poincar\'e self-duality for \(\hi\) begins with the following result.

\begin{thm}[Basic K-cycle]\label{intro: basic K-cycle}
With the above notations, the pair
\begin{align*}
\big(\lambda_\mu, \projcon\big)
\end{align*}
is an odd Fredholm module over \(\hi\). Moreover, $(\lambda_\mu, \projcon)$ is $p$-summable for every $p>\max\{2,\visdim \bd X\}$, and it represents the Poincar\'e dual \(\Delta \cap [1]\in \K^1(\hi)\) of the unit class \([1]\in \K_0(\hi)\). 
\end{thm}

A certain compatibility between the constructions going into Theorem \ref{intro: basic K-cycle} and the Poincar\'e duality of \cite{Emerson} implies that one can `twist' the basic K-cycle above with projections or unitaries in \(\hi\) in a certain way, generalizing Theorem \ref{intro: basic K-cycle} to cover arbitrary K-homology classes -- leading to the following essential fact about \(\K\)-homology classes for \(\hi\).

\begin{thm}[Twisted K-cycles]\label{intro: twisted K-cycles}
Let \(\G\) be regular and torsion-free. Then the following hold.

\begin{itemize}[leftmargin=20pt, itemsep=3pt]
\item Every class in $\K^1(\hi)$ is represented by an odd Fredholm module of the form 
\begin{align*}
\big(\lambda_\mu, P_{\ell^2\G}\lambdaop_\mu(e)P_{\ell^2\G}\big), \qquad e \textrm{ projection in }\hi.
\end{align*}
Moreover, the projection $e\in \hi$ can be chosen so that the Fredholm module is \(p\)-summable for every \(p >\max\{ 2,  \visdim \bd X\}\). 

\item Every class in $\K^0(\hi)$ is represented by a balanced even Fredholm module of the form
\begin{align*}
\big(\lambda_\mu, P_{\ell^2\G}\lambdaop_\mu(u)P_{\ell^2\G}+(1-P_{\ell^2\G})\big), \qquad u \textrm{ unitary in }\hi.
\end{align*}
Moreover, the unitary $u\in \hi$ can be chosen so that the Fredholm module is \(p\)-summable for every \(p >\max\{ 2,  \visdim \bd X\}\). 
\end{itemize}
\end{thm}

Here $\lambdaop_\mu$ is the \emph{right regular representation} of $\hiX$ on $\ell^2(\G,L^2(\bd X,\mu))$, the conjugate of $\lambda_\mu$ by an appropriate self-adjoint unitary \(J\colon \hiX \to \hiX\).

While all projections and all unitaries in $\hi$ yield Fredholm modules as above, one needs to restrict to a suitable smooth subalgebra in order to get finite summability. For each $p$, the twisted K-cycles are $p$-summable over one and the same dense $*$-subalgebra, the algebraic crossed-product $\Lip(\bd X, d)\rtimes_{\textup{alg}}\Gamma$ where $d$ is a visual metric on $\bd X$ of Hausdorff dimension at most $p$. We thus obtain Theorem A, in the following more precise form.

\begin{thm}[Uniform summability]\label{intro: uniform for boundary}
Let $\G$ be regular and torsion-free. Then the K-homology of $\hi$ is uniformly $p$-summable for every $p >\max\{ 2,  \visdim \bd X\}$.
\end{thm}

We now turn our attention to the K-homology of the reduced C$^*$-algebra $C^*_r\G$. Here issues related to the Baum - Connes conjecture mean that, in general, our methods only yield results about the `\(\gamma\)-part' of the K-homology of \(C^*_r\G\). The key tool is the following Gysin sequence, which computes the restriction map $i^*\colon \K^*(\hi) \to \gamma \K^*(C^*_r\G)$ on K-homology induced by the inclusion \(i: C^*_r\G\to \hi\). This sequence is the K-homology version of the one in \cite{Emerson:Euler}, which computes the map \(i_*\colon \K_*(C^*_r\G) \to \K_*(\hi)\) induced by \(i\) on K-theory.

\begin{thm}[Gysin sequence for K-homology]\label{intro: gysin} 
Let \(\G\) be torsion-free. Let \(\gamma \in \KK^\G_0(\C, \C) \) be the \(\gamma\)-element for \(\G\), and \(\gamma \K^*(C^*_r\G)\) the corresponding summand of the K-homology of \(C^*_r\G\). Then there is an exact sequence 
%\begin{multline*}
%0 \to \K_1(B\G)\rightarrow \K^0(\hi) \xrightarrow{i^*} \gamma\KK^\G_0(\C, \C) \xrightarrow{\Eul} \K_0(B\G)\\ \to \K^1(\hi)\xrightarrow{i^*} \gamma\KK^\G_1(\C, \C) \to 0
%\end{multline*}
\begin{align*}
\minCDarrowwidth10pt\
\begin{CD}
0 @>>> \K_1(B\G) @>>> \K^0(\hi) @>i^*>> \gamma\K^0(C^*_r\G)  @. \\ @.@. @. @VV\Eul V\\
@. @. @.  \K_0(B\G) @>>> \K^1(\hi) @>i^*>> \gamma\K^1(C^*_r\G) @>>> 0
\end{CD}
\end{align*}
where \(\Eul\) is the map $\Eul (a) = \chi (\G)\: \ind(a)\:[\pnt]\in \K_0(B\G)$, and where 
\(\ind\) is the ordinary Fredholm index map \(\KK^\G(\C, \C) \to \Z\), \([\pnt]\) the \(\K\)-homology class of a 
point. 
\end{thm}

We note that the torsion assumption could be dropped, at the expense of elaborating the sequence in the 
way that was done in \cite{Emerson:Euler}. The Gysin sequence, combined with Theorem~\ref{intro: twisted K-cycles}, yields the following.

\begin{thm}[Twisted K-cycles over the reduced C*-algebra]\label{reps for reduced}
Let $\G$ be regular and torsion-free. Then the following hold.

\begin{itemize}[leftmargin=20pt, itemsep=3pt]
\item Every class in the \(\gamma\)-part $\gamma \K^1(C^*_r\G)$ of the \emph{odd} \(\K\)-homology of \(C^*_r\G\)
 is represented by an odd Fredholm module of the form
\begin{align*}
\big(\lambda, P_{\ell^2\G}\lambdaop_\mu(e)P_{\ell^2\G}\big), \qquad e \textrm{ projection in }\hi.
\end{align*}
Moreover, the projection $e\in \hi$ can be chosen so that the Fredholm module is \(p\)-summable over $\C\G$ for every \(p >\max\{ 2,  \visdim \bd X\}\).

\item If $\chi (\G) = 0$, then, similarly, every class in the \(\gamma\)-part $\gamma \K^0(C^*_r\G)$ is represented by a balanced even Fredholm module of the form
\begin{align*}
\big(\lambda, P_{\ell^2\G}\lambdaop_\mu(u)P_{\ell^2\G}\big),\qquad u \textrm{ unitary in }\hi.
\end{align*}
Moreover, the unitary $u\in \hi$ can be chosen so that the Fredholm module is \(p\)-summable over $\C\G$ for every \(p >\max\{ 2,  \visdim \bd X\}\). 

\item \noindent If $\chi (\G) \neq 0\) and if \(\gamma_r\in \gamma \K^0(C^*_r\G)\) is a reduced \(\gamma\)-element, then every class in the \(\gamma\)-part $\gamma \K^0(C^*_r\G)$ is, up to an integral multiple of $\gamma_r$, represented by a balanced even Fredholm module as above.
\end{itemize}
\end{thm}

Here $\lambda$ denotes, as usual, the regular representation of $\G$. A `reduced' \(\gamma\)-element is roughly the same as a 
\(\gamma\)-element (a class in \(\K^0(C^*\G)\) which factors in a certain way), but one which is defined over \(C^*_r\G\) rather 
than \(C^*\G\).

We do not know whether, in general, there exists a reduced \(\gamma\)-element with a finitely summable representative, for general hyperbolic groups. We also do not know whether, in general, the $\gamma$-element acts as the identity on $\K^*(C^*_r\G)$. But for the class of a-T-menable groups we do know, thanks to Higson - Kasparov \cite{HK}, that \(\gamma = 1\). Specializing Theorem~\ref{reps for reduced} to this class, we obtain:

\begin{thm}[Uniform summability for a-T-menable groups]\label{bigone applied}
Assume that $\G$ is regular, torsion-free, and a-T-menable. Then the odd K-homology $\K^1(C^*_r\G)$ is uniformly $p$-summable over $\C\G$ for every $p >\max\{ 2,  \visdim \bd X\}$. If  $\chi (\G) = 0$, then the even K-homology $\K^0(C^*_r\G)$ is uniformly $p$-summable over $\C\G$ for every $p >\max\{ 2,  \visdim \bd X\}$. If the \(\gamma\)-element $\gamma_r \in \K^0(C^*_r\G)$ can be represented by a $p(\gamma_r)$-summable Fredholm module over $\C\G$, then the even K-homology $\K^0(C^*_r\G)$ is uniformly $p$-summable over $\C\G$ for every $p >\max\{ 2, \visdim \bd X, p(\gamma_r)\}$. 
\end{thm}

A-T-menable hyperbolic groups include finitely generated free groups, cocompact lattices in $\SO(n,1)$ and $\SU(n,1)$, and $C'(1/6)$ small-cancellation groups -- the latter by \cite{Wise}. Applying the previous theorem to each one of these classes, we obtain the following consequences.

\begin{cor}
Let $\G$ be a finitely generated free group. Then the K-homology of $C^*_r\G$ is uniformly $p$-summable over $\C \G$ for every $p>2$.
\end{cor}

\begin{cor}\label{rank-1} If $\G$ is a torsion-free cocompact lattice in $\SO(n,1)$, then the K-homology of $C^*_r\G$ is uniformly $n^+$-summable over \(\C \G\) when $n\geq 3$, respectively $p$-summable over $\C\G$ for every $p>2$, when $n=2$. If $\G$ is a torsion-free cocompact lattice in \(\SU(n,1)\), then the K-homology of  $C^*_r\G$ is uniformly $(2n)^+$-summable over $\C \G$.
\end{cor}

These two corollaries rely on the existence of finitely summable representatives for the $\gamma$-element, due to Julg - Valette \cite{JV} in the free group case, Kasparov \cite{Kas1} in the $\SO(n,1)$ case, respectively Julg - Kasparov \cite{JK} in the $\SU(n,1)$ case. For small-cancellation groups, the outcome is less satisfactory. We have to apply the vanishing Euler characteristic criterion of Theorem~\ref{bigone applied}, as we are lacking information on the finite summability of the $\gamma$-element, and we also do not have an explicit formula for the visual dimension.

\begin{cor}\label{one-sixth-intro} Let $\G$ be a torsion-free group given by a $C'(1/6)$ presentation $\la S\: | \: \mathcal{R}\ra$. Then the odd K-homology $\K^1(C^*_r\G)$ is uniformly summable over $\C\G$, and the same is true for the even K-homology $\K^0(C^*_r\G)$ provided that $|S|-|\mathcal{R}|=1$.
\end{cor}

\noindent\textbf{Further problems.} Our work calls attention to the following questions. The first two have already been mentioned.

\begin{itemize}[leftmargin=20pt,itemsep=3pt]
\item Let $\G$ be a hyperbolic group. Does there exist a finitely summable representative for a reduced $\gamma$-element $\gamma_r\in \K^0(C^*_r\G)$?

\item Let $\G$ be a hyperbolic group. Does $\gamma$ act as the identity on the K-homology of $C^*_r\G$? 

\item
Let $\G$ be a hyperbolic group. Is the K-homology of $C^*_r\G$ uniformly summable over $\C\G$? 
\end{itemize}

We conjecture the answer to the first problem to be positive for fundamental groups of compact, negatively curved manifolds. It is quite plausible for the answer to be positive in general. Note that our results cover the case when $\chi(\G)=0$. 

The second problem may well have a negative answer.

Positive solutions to the first two problems also settle the third. But it is entirely possible that the third problem can be attacked from a different perspective. 

%%%%%%%%%%%%%%%%%%%%%%%%%%%%%%%%%%%%%%%%%%%%%%%%%%%%%%%%%%%%%%%%%%%%%%%

\section{Preliminaries on K-homology} 
\subsection{Fredholm modules and K-homology} We recall some definitions, while fixing notations along the way. For further details, we refer to Connes \cite[Ch.IV]{Connes} and Higson - Roe \cite[Ch.8]{HR}.

Let $A$ be a unital C*-algebra. As usual, $\mathcal{B}(H)$ and $\mathcal{K}(H)$ denote the bounded operators, respectively the ideal of compact operators on a (separable) Hilbert space $H$.

\begin{defn}[Atiyah, Kasparov]\label{def:FM} An \emph{odd Fredholm module} for $A$ is a pair $(\pi,P)$, where $\pi:A\to \mathcal{B}(H)$ is a representation, $P: H\to H$ is an essential projection in the sense that $P^*-P, P^2-P\in \mathcal{K}(H)$, and such that $[P,\pi(a)]=P\pi(a)-\pi(a)P\in \mathcal{K}(H)$ for all $a\in A$.

An \emph{even Fredholm module} for $A$ is a pair $(\pi_\pm,U)$, where $\pi_\pm:A\to \mathcal{B}(H_\pm)$ are representations, $U:H_+\to H_-$ is an essential unitary in the sense that $U^*U-1\in \mathcal{K}(H_+), UU^*-1\in  \mathcal{K}(H_-)$, and such that $\pi_+(a)-U^*\pi_-(a)U\in \mathcal{K}(H_+)$ for all $a\in A$. If $H_+=H_-=:H$ and $\pi_+=\pi_-=: \pi$, then we say that the even Fredholm module is \emph{balanced}, and we simply write it $(\pi,U)$.
\end{defn}

Fredholm modules are the cycles in Kasparov's K-homology groups, and for that reason they are also called \emph{K-cycles}. Here is an outline of the odd case, leading to the odd K-homology group $\K^1(A)$. The equivalence relation defined by Kasparov on odd Fredholm modules is generated by \emph{unitary equivalence}, \emph{operator homotopy}, and \emph{addition of degenerates}. Unitary equivalence has the obvious meaning. Two Fredholm modules $(\pi,P_0)$ and $(\pi,P_1)$ are operator homotopic if there is a norm-continuous path of essential projections $(P_t)_{t\in [0,1]}$ such that $(\pi,P_t)$ is a Fredholm module at all times $t\in [0,1]$. Thus, the representation $\pi$ is fixed throughout an operator homotopy. A Fredholm module $(\pi,P)$ is degenerate if $P$ is a projection which commutes with the representation $\pi$. Under direct summation of K-homology classes, $\K^1(A)$ is an abelian group. Modulo essentially the same equivalence relation as in the odd case, even Fredholm  modules up to equivalence are the classes in the even K-homology group \(\K^0(A)\). Every class in $\K^0(A)$ can be represented by a balanced even Fredholm module.

\subsection{Finitely summable Fredholm modules} The singular values $\{s_n(T)\}_{n\geq 1}$ of a compact operator $T\in\mathcal{K}(H)$ are the eigenvalues of $|T|$, arranged in non-increasing order and repeated according to their multiplicity. The compactness of $T$ means that $s_n(T)\to 0$. For $p\geq 1$, the Schatten ideals $\mathcal{L}^p(H)$ and $\mathcal{L}^{p+}(H)$ are defined as follows:
\begin{align*}
\mathcal{L}^p(H)=\Big\{T\in \mathcal{K}(H):\; \sum s_n(T)^p<\infty\Big\}, \qquad\mathcal{L}^{p+}(H)=\Big\{T\in \mathcal{K}(H):\; s_n(T)=O(n^{-1/p})\Big\}.
\end{align*}
(Actually, the definition of $\mathcal{L}^{1+}(H)$ is slightly different, and it will not be used in this paper.) We have $\mathcal{L}^p(H)\subset\mathcal{L}^{p+}(H)\subset\mathcal{L}^q(H)$ for all $q>p$.

The summable Fredholm modules are those which satisfy a restricted version of  Definition~\ref{def:FM}, in which the ideals $\mathcal{L}^p(H)$ or $\mathcal{L}^{p+}(H)$ replace the ideal of compact operators $\mathcal{K}(H)$. 

\begin{defn}[Connes]
An odd Fredholm module $(\pi,P)$ is \emph{$p$-summable} (over $\mathcal{A}$) if $P^*-P, P^2-P\in \mathcal{L}^p(H)$ and $[P, \pi(a)]\in\mathcal{L}^p(H)$ for all $a$ in a dense subalgebra $\mathcal{A}$ of $A$. A balanced even Fredholm module $(\pi,U)$ is \emph{$p$-summable} (over $\mathcal{A}$) if $UU^*-1, U^*U-1 \in \mathcal{L}^p(H)$, and $[U,\pi(a)]\in \mathcal{L}^p(H)$ for all $a$ in a dense subalgebra $\mathcal{A}$ of $A$.
\end{defn}

The notion of $p^+$-summable Fredholm module is defined analogously. We note that $p$-summability implies $p^+$-summability, which in turn implies $q$-summability for all $q>p$.

The property that every K-homology class is representable by a finitely summable Fredholm module could be deemed as \emph{K-homological finiteness}. Even sharper would be to require that finite summability can be achieved in a uniform way throughout the K-homology classes. Such a uniformity could be imposed on the degree of summability, or on the summability subalgebra, or both. We propose the following definition.

\begin{defn}
The K-homology of a C*-algebra $A$ is \emph{uniformly $p$-summable (over $\mathcal{A}$)} if there is a dense subalgebra $\mathcal{A}$ of $A$ such that every K-homology class of $A$ can be represented by a Fredholm module which is $p$-summable over $\mathcal{A}$. 
\end{defn}

There is an obvious variation for $p^+$-summability. The motivating example for this strong notion of K-homological finiteness is the following: the K-homology of the commutative C*-algebra $C(M)$, where $M$ is a smooth closed manifold, is uniformly $(\dim M)^+$-summable over the smooth subalgebra $C^\infty(M)$.

%%%%%%%%%%%%%%%%%%
\section{The basic K-cycle}\label{paradigm}
Throughout this section, $G$ is a discrete countable group acting by homeomorphisms on a compact metrizable space $X$. To avoid trivialities, we assume that $X$ is not a singleton. We consider the reduced crossed-product $C(X)\rcross G$ associated to the topological dynamics $G\act X$.

\subsection{Left regular representation, $G$-expectation and $G$-deviation} Let $\mu$ be a Borel probability measure on $X$ which has full support, meaning that non-empty open subsets have positive measure, and it is $G$-quasi-invariant, i.e., the action of $G$ preserves the measure class of $\mu$. We do not assume $\mu$ to be $G$-invariant, in fact a highly non-invariant $\mu$ will turn out to be the most interesting case for our purposes.

The faithful representation of $C(X)$ on $L^2(X,\mu)$ by multiplication induces a faithful representation $\lambda_\mu$ of $C(X)\rcross G$ on $\ell^2(G, L^2(X,\mu))$, the \emph{left regular representation with respect to $\mu$}. In fact, the C*-algebra $C(X)\rcross G$ can be defined as the norm completion of the algebraic crossed-product $C(X)\rtimes_{\mathrm{alg}} G$ in the regular representation $\lambda_\mu$. Concretely, $\lambda_{\mu}$ is given as follows:
\begin{align*}
\lambda_{\mu}(\phi)\Big(\sum \psi_h \delta_h\Big)=\sum (h^{-1}.\phi)\psi_h \delta_h, \qquad \lambda_{\mu}(g)\Big(\sum \psi_h \delta_h\Big)=\sum \psi_h \delta_{gh}
\end{align*}
where $\phi\in C(X)$, $g\in G$, and $\sum \psi_h \delta_h\in \ell^2(G, L^2(X,\mu))$. The covariance relation $\lambda_\mu(g.\phi)=\lambda_\mu(g)\lambda_\mu(\phi)\lambda_\mu(g^{-1})$ holds. 

On the probability space $(X,\mu)$, momentarily devoid of the $G$-action, there are two important numerical characteristics attached to a continuous functions on $X$: the \emph{expectation} and the \emph{standard deviation}. Namely, for $\phi\in C(X)$ we put
\begin{align*}\E \phi=\int \phi \: d\mu,\qquad \dev  \phi=\sqrt{\E\big(|\phi|^2\big)-\big|\E\phi\big|^2}.
\end{align*}
When we bring in the $G$-action, we are led to consider the following dynamical counterparts.

\begin{defn}
The \emph{$G$-expectation} and the \emph{$G$-deviation} of $\phi\in C(X)$ with respect to $\mu$ are the functions $\E\phi:G\to\C$ and $\dev\phi:G\to [0,\infty)$ given as follows:
\begin{align*}
\E\phi(g)=\int g^{-1}.\phi \: d\mu=\int \phi \: d g_*\mu,\qquad \dev \phi=\sqrt{\E \big(|\phi|^2\big)-\big|\E\phi\big|^2}.
\end{align*}
\end{defn}

An explicit, and useful, formula for the $G$-deviation is
\begin{align}\label{iint}\dev \phi(g)=\sqrt{\tfrac{1}{2}\iint |\phi(gx)-\phi(gy)|^2 \: d\mu(x)\: d\mu(y)}.
\end{align}
As an illustration of the dynamical expectation for a non-trivial group action, consider the case of a group $G\subseteq \mathrm{SU}(1,1)$ acting by linear fractional transformations on the unit circle $S^1=\{\zeta\in \C:|\zeta|=1\}$. With respect to the normalized Lebesgue measure, $\E\phi(g)$ is the value of the Poisson transform of $\phi$ on the unit disk at the point $g(0)$.

\begin{defn}
We say that $\mu$ has \emph{$C_0$-deviation} if $\dev \phi\in C_0(G)$ for all $\phi\in C(X)$, respectively \emph{$\ell^p$-deviation} if $\dev \phi\in \ell^p G$ for all $\phi$ in a dense subalgebra of $C(X)$. 
\end{defn}

\subsection{The basic K-cycle} We view $\ell^2 G$ as the constant-coefficient subspace of $\ell^2(G, L^2(X,\mu))$. The corresponding projection $P_{\ell^2 G}$ is given by coefficient-wise integration: 
\begin{align*}
P_{\ell^2 G}\Big(\sum \psi_h \delta_h\Big)=\sum \bigg(\int \psi_h \: d\mu\bigg) \delta_h.
\end{align*}

We are interested in the event that $(\lambda_\mu, P_{\ell^2 G})$ is a Fredholm module -- or, even better, a summable one -- for $C(X)\rcross G$. When this happens, we refer to $(\lambda_\mu, P_{\ell^2 G})$ as \emph{the basic K-cycle associated to $\mu$}. The Fredholmness and the summability of $(\lambda_\mu, P_{\ell^2 G})$ can be conveniently expressed in terms of the decay of the $G$-deviation. First, we record a general observation regarding Fredholmness and summability in the odd case.

\begin{lem}\label{howtocheckFredholmness}
Let $A$ be a unital C*-algebra, let $\pi:A\to \mathcal{B}(H)$ be a representation, and let $P$ be a projection in $\mathcal{B}(H)$. Denote by $s(a):=P\pi(a)P$ the corresponding compression. Then $(\pi,P)$ is a Fredholm module for $A$ if and only if $\sqrt{s(|a|^2)-|s(a)|^2}\in \mathcal{K}(H)$ for all $a\in A$. Furthermore, $(\pi,P)$ is $p$-summable over a dense $*$-subalgebra $\mathcal{A}\subseteq A$ if and only if $\sqrt{s(|a|^2)-|s(a)|^2}\in \mathcal{L}^p(H)$ for all $a\in \mathcal{A}$.
\end{lem}

\begin{proof}
Let $\Pi(a)=(1-P)\pi(a)P$; this is the lower left corner of the $2$-by-$2$ matrix defined by  the decomposition of $\pi$ with respect to $P$. Using the relations 
\begin{align*}
 %\label{crit!}
 [\pi (a), P]= \Pi(a)-\Pi(a^*)^*, \qquad \Pi(a)= (1-P) [\pi (a), P]
 \end{align*}
we see that $(\pi,P)$ is a Fredholm module if and only if $\Pi(a)\in \mathcal{K}(H)$ for all $a\in A$, and that $(\pi,P)$ is $p$-summable over $\mathcal{A}$ if and only if $\Pi(a)\in \mathcal{L}^p(H)$ for all $a\in \mathcal{A}$. Now
\begin{align*}
\Pi(a)^*\Pi(a)&=P\pi(a^*)(1-P)\pi(a)P\\
&=P\pi(a^*a)P-\big(P\pi(a)^*P\big)\big(P\pi(a)P\big)=s(a^*a)-s(a)^*s(a)
\end{align*}
shows that $|\Pi(a)|=\sqrt{s(|a|^2)-|s(a)|^2}$. \end{proof}

\begin{prop}\label{from deviation to Fredholm} The pair $(\lambda_\mu, P_{\ell^2 G})$ is a Fredholm module if and only if $\mu$ has $C_0$-deviation. If $\mu$ has $\ell^p$-deviation, then $(\lambda_\mu, P_{\ell^2 G})$ is a $p$-summable Fredholm module; for $p\geq 2$, the converse holds.
\end{prop}

\begin{proof} The projection $P_{\ell^2G}$ compresses the space restriction $\lambda_{\mu}|_{C(X)}$ to multiplication by the $G$-expectation on $\ell^2G$:
\begin{align}\label{projection compressed lambda}
P_{\ell^2G}\lambda_{\mu}(\phi)P_{\ell^2G}=\M(\E\phi)
\end{align}
for all $\phi\in C(X)$. Hence for $s_\mu(\phi):=P_{\ell^2G}\lambda_{\mu}(\phi)P_{\ell^2G}$ we have
\begin{align*}
\sqrt{s_\mu(|\phi|^2)-|s_\mu(\phi)|^2}=\M(\dev\phi).
\end{align*}
By the proof of Lemma~\ref{howtocheckFredholmness}, $\mu$ has $C_0$-deviation if and only if $[\lambda_{\mu}(\phi), P_{\ell^2G}]$ is compact for all $\phi\in C(X)$. As $P_{\ell^2G}$ commutes with the group restriction $\lambda_{\mu}|_G$, the latter condition is equivalent to having $[\lambda_{\mu}(a), P_{\ell^2G}]$ compact for all $a\in C(X)\rtimes_\mathrm{alg} G$, which is equivalent to $(\lambda_\mu, P_{\ell^2 G})$ being Fredholm. 

The summable analogue is argued in a similar way. For sufficiency, assume that $\mu$ has $\ell^p$-deviation. Then there is a $G$-invariant, dense $*$-subalgebra $A(X)\subseteq C(X)$ such that $\dev \phi\in \ell^p G$ for all $\phi\in A(X)$. As above, we deduce that $[\lambda_{\mu}(a), P_{\ell^2G}]$ is a $p$-summable operator for all $a\in A(X)\rtimes_\mathrm{alg} G$. Thus $(\lambda_\mu, P_{\ell^2 G})$ is $p$-summable.

For the converse, we bring in another expectation, namely the bounded linear map $\mathbb{E}: C(X)\rcross G\onto C(X)$ defined by $\mathbb{E}(\sum \phi_g\: g)=\phi_1$ over $C(X) \rtimes_\mathrm{alg} \G$. We claim that
\begin{align}\label{conditional expectation smaller}
\|\Pi(a)\delta_h\|_2\geq \sigma(\mathbb{E}(a))(h)
\end{align}
for all $h\in G$ and $a\in C(X)\rcross G$. Indeed, using along the way the fact that $\Pi(\phi_{g_2})^*\Pi(\phi_{g_1})$ is a multiplication operator on $\ell^2G$, we have:
\begin{align*}
\big\la \Pi(a) \delta_h,\Pi(a)\delta_h\big\ra&=\sum_{g_1,g_2} \big\la \Pi(\phi_{g_1}) \delta_{g_1h},\Pi(\phi_{g_2})\delta_{g_2h}\big\ra=\sum_{g_1,g_2}  \big\la \Pi(\phi_{g_2})^*\Pi(\phi_{g_1}) \delta_{g_1h},\delta_{g_2h}\big\ra\\
&=\sum_{g}  \big\la \Pi(\phi_{g})^*\Pi(\phi_g) \delta_{gh},\delta_{gh}\big\ra=\sum_{g}  \big\la \M(\dev\phi_g)^2 \delta_{gh},\delta_{gh}\big\ra\\
&=\sum_{g}  (\dev\phi_g)^2(gh)\geq (\dev\phi_1)^2(h)=\big(\sigma(\mathbb{E}(a))(h)\big)^2
\end{align*}
Now assume that $(\lambda_\mu, P_{\ell^2 G})$ is a $p$-summable Fredholm module for $C(X)\rcross G$. Then $\Pi(a)$ is a $p$-summable operator for all $a$ in a dense subalgebra $\mathcal{A}$ of $C(X)\rcross G$. For $p\geq 2$, the $p$-summability of $\Pi(a)$ implies the $p$-summability of $\{\|\Pi(a)\xi_\iota\|_2\}_{\iota\in I}$ for any orthonormal system $(\xi_\iota)_{\iota\in I}$ (\cite[Thm.1.18]{Sim}). In particular $\{\|\Pi(a)\delta_h\|_2\}_{h\in G}$ is $p$-summable, so $\sigma\big(\mathbb{E}(a)\big)\in \ell^p G$ by \eqref{conditional expectation smaller}. Thus, we have shown that $\dev \phi\in \ell^p G$ for all $\phi\in \mathbb{E}(\mathcal{A})$. It follows that $\{\phi\in C(X):\dev \phi\in \ell^p G \}$, which is always a subalgebra of $C(X)$, is dense. We conclude that $\mu$ has $\ell^p$-deviation.
\end{proof}

At the current level of generality, we cannot address the question whether $(\lambda_\mu, P_{\ell^2G})$, when a Fredholm module, is homologically non-trivial or not. It is, however, clear that it is non-degenerate, given our assumptions that $X$ is not a singleton and $\mu$ is fully-supported.

\subsection{$C_0$-deviation and the convergence property}\label{sec: C0 and conv} Let $\mathrm{Prob}(X)$ denote the space of Borel probability measures on $X$, and equip $\mathrm{Prob}(X)$ with the weak$^*$ convergence induced by $C(X)$: by definition, $\nu_\iota\to \nu$ if $\int \phi\: d\nu_\iota\to \int \phi\: d\nu$ for all $\phi\in C(X)$. The space $\mathrm{Prob}(X)$ is compact. In particular, push-forwards of $\mu$ by elements of $G$ must accumulate. We make the following definition.

\begin{defn} The probability measure $\mu$ is said to have the \emph{convergence property} if the accumulation points of the $G$-orbit $G\mu\subseteq\mathrm{Prob}(X)$ are all point measures.
\end{defn}

We think of the convergence property for a probability measure as a measurable analogue of an established notion in topological dynamics, that of a convergence group action. Let us recall the definition, originally due to Gehring and Martin for the case of group actions on spheres or closed balls, and then subsequently extended by Tukia, Freden, Bowditch to the general case of group actions on compact metrizable spaces. The action of $G$ on $X$ is said to be a \emph{convergence action} if the following holds: for each sequence $(g_n)\subseteq G$ with $g_n\to \infty$, there is a subsequence $(g_{n_i})$ with attracting and repelling points $x^+,x^-\in X$ in the sense that $g_{n_i}z\to x^+$ uniformly outside neighbourhoods of $x^-$. See, for instance, \cite[Sec.5]{KB} for further details and references.

We note the following simple fact.

\begin{prop}\label{convergence action}
Let $G\act X$ be a convergence action. If points in $X$ are $\mu$-negligible, then $\mu$ has the convergence property.
\end{prop}

\begin{proof} Let $(g_n)_*\mu$ converge in $\mathrm{Prob}(X)$, where $g_n\to \infty$ in $G$. Without loss of generality $(g_n)$ has attracting and repelling points $x^+,x^-\in X$. We claim that $(g_n)_*\mu\to \delta_{x^+}$. Indeed, let $\phi\in C(X)$. Then $g_n^{-1}.\phi$ converges pointwise to the constant function $\phi(x^+)$ on $X-\{x^-\}$. As $x^-$ is $\mu$-negligible, we get $\int \phi \: d(g_n)_*\mu=\int g_n^{-1}.\phi\: d\mu\to \phi(x^+)$ by Lebesgue's dominated convergence theorem. 
\end{proof}

The convergence property is relevant for our discussion, in light of the following characterization.

\begin{prop}\label{Furstenberg type condition}
The probability measure $\mu$ has $C_0$-deviation if and only if it has the convergence property.
\end{prop}

\begin{proof} Assume that $\mu$ has $C_0$-deviation, and let $\nu\in \mathrm{Prob}(X)$ be the limit of a sequence $(g_n)_*\mu$ with $g_n\to \infty$ in $G$. For each $\phi\in C(X)$ we have, on the one hand, that $\dev\phi(g_n)$ converges to $0$, and on the other hand that $\dev\phi(g_n)$ converges to the standard deviation of $\phi$ with respect to $\nu$. Therefore, $\int |\phi|^2\: d\nu=|\int \phi \: d\nu|^2$ for all $\phi\in C(X)$. This continues to hold throughout $L^2(X,\nu)$, by the density of $C(X)$ in $L^2(X,\nu)$ -- Borel probability measures on compact metrizable spaces are automatically Radon. Taking characteristic functions of measurable sets, we see that $\nu$ is $\{0,1\}$-valued. But the only $\{0,1\}$-valued Borel probability measures on $X$ are the point measures: choosing a compatible metric on $X$, there exists a sequence of full-measure balls with radius converging to $0$, hence a point having full measure.

The converse implication is left to the reader.\end{proof}

In \S \ref{sec: finite summability} we will show that suitable measures 
on the boundary of a Gromov hyperbolic group have the convergence property 
with respect to the boundary action of the group (which is a convergence action.)

\subsection{Double ergodicity and $\ell^p$-deviation} We address the condition $p\geq 2$, encountered in Proposition~\ref{from deviation to Fredholm}. Namely, we show that double ergodicity of $\mu$ is an obstruction to having $\ell^p$-deviation with $p\leq 2$.

\begin{prop}\label{2ergo}
If $\mu\times \mu$ is ergodic for the diagonal action of $G$ on $X\times X$, and $X$ has no isolated points, then a function $\phi\in C(X)$ with $\sigma_G\phi\in \ell^2 G$ must be constant. In particular, if $\mu$ has $\ell^p$-deviation then $p>2$.
\end{prop}

\begin{proof}
Arguing by contradiction, we assume that $\phi\in C(X)$ is a non-constant function with the property that $\sigma\phi\in \ell^2 G$. By \eqref{iint}, we have
\begin{align*}
\|\dev \phi\|^2_{\ell^2G}=\tfrac{1}{2}\iint \sum_{g\in G}|\phi(gx)-\phi(gy)|^2 \:d\mu(x)d\mu(y)
\end{align*}
Therefore $S(x,y)=\sum_{g\in G}|\phi(gx)-\phi(gy)|^2$ defines a $G$-invariant $L^2$ map on $X\times X$. By ergodicity, $S$ is a.e. constant, say $S(x,y)=C$ for almost all $(x,y)\in X\times X$. 

There exists $c>0$ such that the open subset $V=\{(x,y): |\phi(x)-\phi(y)|> c\}\subseteq X\times X$ is non-empty. As $X\times X$ has no isolated points, for each positive integer $N$ there exist disjoint, non-empty open subsets $U_1,\dots, U_N\subseteq V$. Using again the ergodicity assumption, we have that each $G\cdot U_i=\cup_{g\in G}\: gU_i$ is either negligible or of full measure. Since non-empty open subsets of $X\times X$ have positive measure, the latter alternative must occur. It follows that $\cap_{i=1}^N\: G\cdot U_i$ has full measure. Let $(x,y)$ in $\cap_{i=1}^N\: G\cdot U_i$ with $S(x,y)=C$. Thus, for each $i$ we have some $g_i\in G$ such that $(g_ix,g_iy)\in U_i$. Now the $g_i$'s are distinct since the $U_i$'s are disjoint, and $|\phi(g_ix)-\phi(g_iy)|>c$ since $U_i\subseteq V$, so
\[C=S(x,y)\geq \sum_{i=1}^N|\phi(g_ix)-\phi(g_iy)|^2>Nc^2.\]
As $N$ is arbitrary, this is a contradiction.
\end{proof}

%%%%%%%%%%%%%%%%%%%%%%%%%%%%%%%%%%%%
\section{Further properties of the basic K-cycle}\label{paradigm II}
We now investigate the behaviour of the basic K-cycle under two operations: changing the measure $\mu$, respectively passing to a finite-index subgroup of $G$. We keep the notations of the previous section.

\subsection{Comparable measures} The pair
$(\lambda_\mu, P_{\ell^2 G})$ is constructed in reference to the measure $\mu$, which is not part of the given topological setting. Nevertheless, its relevant features -- Fredholmness, degree of summability, K-homology class -- are canonical over the measure class of $\mu$. The suitable equivalence here is the following: a Borel probability measure $\mu'$ on $X$ is said to be \emph{comparable} to $\mu$ if $\mu'\asymp \mu$, in the sense that $C_1\mu\leq \mu'\leq C_2\mu$ for some positive constants $C_1, C_2$. Clearly, comparability is finer than the usual equivalence of measures which, we recall, means that each measure is absolutely continuous with respect to the other. Formula \eqref{iint} shows that comparable measures have comparable $G$-deviations, hence the following:

\begin{prop}\label{deviation for comparable measures}
Let $\mu$ and $\mu'$ be comparable probability measures. Then $(\lambda_\mu, P_{\ell^2 G})$ is a Fredholm module if and only if $(\lambda_{\mu'}, P_{\ell^2 G})$ is a Fredholm module. For $p\geq 2$, $(\lambda_\mu, P_{\ell^2 G})$ is $p$-summable if and only if $(\lambda_{\mu'}, P_{\ell^2 G})$ is $p$-summable.
\end{prop}

Most importantly, basic K-cycles associated to comparable measures define one and the same homology class:

\begin{prop}\label{independence of comparable measures}
Let $\mu$ and $\mu'$ be comparable probability measures having $C_0$-deviation. Then the Fredholm modules $(\lambda_\mu, P_{\ell^2 G})$ and $(\lambda_{\mu'}, P_{\ell^2 G})$ are K-homologous.
\end{prop}

\begin{proof}
Let $\rho=d\mu'/d\mu$ be the Radon-Nikodym derivative, so $\rho$ is essentially bounded from above and from below by the comparability constants of $\mu$ and $\mu'$. First, we have a unitary
\begin{align*}
U: \ell^2(G, L^2(X, \mu')) \to \ell^2(G, L^2(X, \mu)), \qquad \sum \psi_h\delta_h\mapsto \sum \sqrt{\rho}\: \psi_h\delta_h
\end{align*}
which intertwines the corresponding regular representations of $C(X)\rcross G$, that is, $U \lambda_{\mu'}U^*=\lambda_{\mu}$. We may therefore exchange $(\lambda_{\mu'}, P'_{\ell^2 G})$ for $(\lambda_\mu, UP'_{\ell^2G}U^*)$, where the notation $P'_{\ell^2G}$ is used in order to emphasize the dependence on $\mu'$. We now claim that the Fredholm modules $(\lambda_\mu, UP'_{\ell^2G}U^*)$ and $(\lambda_\mu, P_{\ell^2G})$ are operator homotopic. Note that
\begin{align*}
UP'_{\ell^2G}U^*\big(\sum \psi_h\delta_h\big)= \sum \sqrt{\rho}\: \bigg(\int \sqrt{\rho}\:\psi_h \: d\mu\bigg)\delta_h,
\end{align*}
and that $\sqrt{\rho}\in L^\infty (X,\mu)$ with $\|\sqrt{\rho}\|_{L^2(X,\mu)}=1$. For $\eta\in L^\infty (X,\mu)$ satisfying $\|\eta\|_{L^2(X,\mu)}=1$, let $M(\eta)$ be the corresponding multiplication operator on $\ell^2(G, L^2(X, \mu))$. Then
\begin{align*}
P(\eta)= M(\bar{\eta})P_{\ell^2G}M(\eta),\qquad \sum \psi_h\delta_h\mapsto \sum \bar{\eta}\: \bigg(\int \eta\:\psi_h \: d\mu\bigg)\delta_h.
 \end{align*}
is a projection, namely the projection of $\ell^2(G, L^2(X, \mu))$ onto $M(\bar{\eta})\ell^2G$. We have $[P(\eta),\lambda_{\mu}]=M(\bar{\eta})[P_{\ell^2G},\lambda_{\mu}]M(\eta)$ since $M(\eta)$ and $M(\bar{\eta})$ commute with $\lambda_{\mu}$, so $(\lambda_\mu, P(\eta))$ is a Fredholm module. On the other hand, we have $\|P(\eta_1)-P(\eta_2)\|\leq 2\|\eta_1-\eta_2\|_{L^2(X,\mu)}$; this follows from the fact that
\begin{align*}
\bigg\|\bar{\eta}_1\int \eta_1 \psi\: d\mu - \bar{\eta}_2\int \eta_2 \psi\: d\mu\bigg\|_2\leq 2\|\eta_1-\eta_2\|_2 \|\psi\|_2
\end{align*}
for all $\psi\in L^2(X, \mu)$. Now let $\eta(t)=(\cos t)\: 1+(i \sin t)\: \sqrt{\rho}$, where $0\leq t\leq\pi/2$. Then $\eta(t)\in L^\infty (X,\mu)$, and $\eta(t)$ describes a continuous path in the unit sphere of $L^2(X,\mu)$ between the constant function $1$ and $i\sqrt{\rho}$. Consequently, $P(\eta(t))$ describes a norm-continuous path between $P(1)=P_{\ell^2G}$ and $P(i\sqrt{\rho})=P(\sqrt{\rho})=UP'_{\ell^2G}U^*$.
\end{proof}

\subsection{Finite-index subgroups}
Let $H\leq G$ be a subgroup. Restriction of representations from $C(X)\rcross G$ to $C(X)\rcross H$ takes Fredholm modules for $C(X)\rcross G$ to Fredholm modules for $C(X)\rcross H$, and it defines a natural homomorphism of abelian groups
\begin{align*}
\mathrm{res}: \K^1(C(X)\rcross G)\to \K^1(C(X)\rcross H).
\end{align*}
Assume that $\mu$ has $C_0$-deviation. On the one hand, restricting $(\lambda^G_\mu, P_{\ell^2G})$ yields a Fredholm module for $C(X)\rcross H$. On the other hand, we can form the Fredholm module $(\lambda^H_\mu, P_{\ell^2H})$ for $C(X)\rcross H$. The homological relation between these two Fredholm modules for $C(X)\rcross H$ is particularly simple in the case when $H$ has finite index in $G$.

\begin{prop}\label{finite index} Assume that $\mu$ has $C_0$-deviation, and that $\{g_*\mu\}_{g\in G}$ forms a family of mutually comparable measures. If $H$ is a finite-index subgroup of $G$, then
\begin{align*}
\mathrm{res}\:  \big[(\lambda^G_\mu, P_{\ell^2G})\big] =[G:H] \: \big[(\lambda^H_\mu, P_{\ell^2H})\big]
\end{align*}
in $\K^1(C(X)\rcross H)$.
\end{prop}

\begin{proof} Put $n=[G:H]$, and pick a transversal $t_1,\dots,t_n$ for the right $H$-cosets. The coset decomposition $\ell^2(G, L^2(X, \mu))=\oplus_{1}^n\: \ell^2(H t_i, L^2(X, \mu))$ yields
\begin{align*}
\mathrm{res}\:  \big[(\lambda^G_\mu, P_{\ell^2G})\big]=\oplus_{1}^n \big[(\lambda_{t_i}, P_{\ell^2(H t_i)})\big]
\end{align*}
in $\K^1(C(X)\rcross H)$, where $\lambda_{t_i}$ denotes the representation of $C(X)\rcross H$ on $\ell^2(H t_i, L^2(X, \mu))$. Now consider $(\lambda_t, P_{\ell^2(H t)})$ for $t\in \{t_1,\dots, t_n\}$. The unitary 
\begin{align*}
R_t: \ell^2(H, L^2(X, \mu))\to \ell^2(H t, L^2(X, \mu)), \qquad \sum \psi_h\delta_h\mapsto \sum \psi_h\delta_{ht}
\end{align*}
implements an equivalence between $(\lambda_t, P_{\ell^2(H t)})$ and $(R_t^*\lambda_t R_t, P_{\ell^2H})$. The representation $R_t^*\lambda_t R_t$ on $\ell^2(H, L^2(X,\mu))$ is given by
\begin{align*}
R_t^*\lambda_t R_t(\phi)\Big(\sum \psi_h \delta_h\Big)=\sum t^{-1}.(h^{-1}.\phi)\psi_h \delta_h, \qquad R_t^*\lambda_t R_t(h')\Big(\sum \psi_h \delta_h\Big)=\sum \psi_h \delta_{h'h}
\end{align*}
for $\phi\in C(X)$ and $h'\in H$. Next, the unitary 
\begin{align*}
V_t: \ell^2(H, L^2(X, \mu))\to \ell^2(H, L^2(X, t_*\mu)), \qquad \sum \psi_h\delta_h\mapsto \sum (t.\psi_h)\delta_h
\end{align*}
makes $(R_t^*\lambda_t R_t,  P_{\ell^2H})$ and $(\lambda^H_{t_*\mu},  P_{\ell^2H})$ equivalent. On the other hand, the assumption that $\{g_*\mu\}_{g\in G}$ consists of mutually comparable measures implies, in light of Proposition~\ref{independence of comparable measures}, that $(\lambda^H_{t_*\mu},  P_{\ell^2H})$ and $(\lambda^H_{\mu},  P_{\ell^2H})$ are homologous. Summarizing, we have
\begin{align*}
\big[(\mathrm{res}(\lambda^G_\mu), P_{\ell^2G})\big]=\oplus_{1}^n\: \big[(\lambda^H_\mu, P_{\ell^2H})\big]\end{align*}
in $\K^1(C(X)\rcross H)$, as desired.
\end{proof}

%%%%%%%%%%%%%%%%%%%%%%%%%%%%%%%%%%%%%%%%%%%%%%%%%%%%%%%%%%%%%%%%%%%%%%%%%%
\section{Preliminaries on boundaries of hyperbolic spaces}\label{Sec: metric-measure}
This section is devoted to the metric-measure structure on the boundary of a hyperbolic space in the sense of Gromov \cite{Gromov}. In \S\ref{bhs} we recall some basic facts on hyperbolic spaces and their boundaries. In \S\ref{vm} we focus on the family of visual metrics, and their induced Hausdorff measures, on the boundary of a hyperbolic space. The content of these two subsections is, to a large extent, standard \cite{GdH}, \cite[Sec.5]{Vai}. In \S\ref{gga} we discuss geometric group actions, and their boundary effect. We describe results of Coornaert \cite{Coo} providing remarkable finiteness properties for the visual metric-measure structure. In \S\ref{visual section} we introduce a suitable notion of `Hausdorff dimension' for the boundary. 

\subsection{The boundary of a hyperbolic space}\label{bhs} Let $(X,d)$ be a proper geodesic space. The \emph{Gromov product} of  $x,y\in X$ with respect to $o\in X$ is defined by the formula
\begin{align*}
(x,y)_o: =\tfrac{1}{2}\big(d(o,x)+d(o,y)-d(x,y)\big).
\end{align*}
\begin{defn}[Gromov]
The space $X$ is \emph{hyperbolic} if there exists a constant $\delta\geq 0$ such that, for all $x,y,z,o\in X$, we have
\begin{align*}
(x,y)_o\geq \min\big\{(x,z)_o,(y,z)_o\big\}-\delta.
\end{align*}
\end{defn}

 Let $X$ be a hyperbolic space, and fix a basepoint $o\in X$. A sequence $(x_i)\subseteq X$ \emph{converges to infinity} if $(x_i,x_j)_o\to\infty$ as $i,j\to\infty$. Two sequences $(x_i)$, $(y_i)$ converging to infinity are \emph{asymptotic} if $(x_i,y_i)_o\to\infty$ as $i\to\infty$. The asymptotic relation is an equivalence on sequences converging to infinity. Both convergence to infinity, and the asymptotic relation, are independent of the chosen basepoint $o\in X$. The \emph{boundary} of $X$, denoted $\bd X$, is the set of asymptotic classes of sequences converging to infinity. A sequence $(x_i)\subseteq X$ \emph{converges to} $\xi\in\bd X$ if $(x_i)$ converges to infinity, and the asymptotic class of $(x_i)$ is $\xi$. 

The Gromov product on $\bd X\times \bd X$ is defined as follows:
\begin{align*}
(\xi,\xi')_o:=\inf \big\{\liminf\; (x_i,x'_i)_o: \; x_i\to\xi,\; x'_i\to\xi'\big\}
\end{align*}
If $\xi=\xi'$, then $(\xi,\xi')_o=\infty$. If $\xi\neq \xi'$, then the sequence $(x_i,x'_i)_o$ is bounded whenever $x_i\to\xi$ and $x'_i\to\xi'$, hence $(\xi,\xi')_o<\infty$. It turns out that 
\begin{align}\label{liminf limsup}
(\xi,\xi')_o\leq \liminf\; (x_i,x'_i)_o\leq \limsup\; (x_i,x'_i)_o \leq (\xi,\xi')_o+ 2\delta \qquad (x_i\to\xi, x'_i\to\xi').
\end{align}
Similarly, the Gromov product on $X\times \bd X$ is defined by setting
\begin{align*}
(x,\xi)_o:=\inf \big\{\liminf\; (x,x_i)_o: \; x_i\to\xi\big\}
\end{align*}
and we have
\begin{align}\label{liminf limsup 2}
(x,\xi)_o\leq \liminf\; (x,x_i)_o\leq \limsup\; (x,x_i)_o \leq (x,\xi)_o+ \delta\qquad (x_i\to\xi). 
\end{align}

\subsection{Visual metrics}\label{vm} Equipped with a canonical topology defined in terms of the Gromov product, the boundary $\bd X$ is compact and metrizable (see \cite[Ch.7, \S2]{GdH}). But the metric structure on $\bd X$, which is of great importance in this paper, is a more subtle issue. 

\begin{defn} A \emph{visual metric} on $\bd X$ is a metric $d_\e$ satisfying $d_\e\asymp\exp(-\e(\cdot,\cdot)_o)$ for some $\e>0$, called the \emph{visual parameter} of $d_\e$.
\end{defn}

This definition is independent of the chosen basepoint $o\in X$, and every visual metric determines the canonical topology on $\bd X$. If $d_\e$ is a visual metric, then so is $d_\e^\alpha$ for any $\alpha\in (0,1]$; consequently, if $\e$ is a visual parameter then so is any $\e'\in (0,\e]$. 

\begin{fact}[Scaling]\label{scaling lemma}
Let $d_\e$ and $d_{\e'}$ be two visual metrics. Then:
\begin{itemize}[leftmargin=20pt, itemsep=3pt]
\item[$\bullet$] $d_\e$ and $d_{\e'}$ are H\"older equivalent: $d^{1/\e}_{\e}\asymp  d^{1/\e'}_{\e'}$;
\item[$\bullet$] the corresponding Hausdorff dimensions are inversely proportional to the visual parameter: $\e \:\hdim(\bd X,d_\e)=\e' \:\hdim (\bd X,d_{\e'})$;
\item[$\bullet$] the corresponding Hausdorff measures are comparable: $\mu_\e \asymp\mu_{\e'}$.
\end{itemize}
\end{fact}

Visual metrics do exist, provided that the visual parameter is small with respect to $1/\delta$ where $\delta$ is a constant of hyperbolicity. Furthermore, there is a companion metric-like map on $X\times \bd X$, which is visual in the corresponding way \cite[Prop.5.16]{Vai}:

\begin{fact}[Small visual range]\label{small visual range} Let $\e>0$ be such that $\e\delta<1/5$. Then:
\begin{itemize}[leftmargin=20pt, itemsep=3pt]
\item[$\bullet$] there exists a visual metric $d_\e$ on $\bd X$, having visual parameter $\e$;
\item[$\bullet$] there exists $d_\e :X\times \bd X\to [0,\infty)$ satisfying 
\begin{align*}\tfrac{1}{2}\exp(-\e(x,\xi)_o)\leq d_\e(x, \xi) \leq \exp(-\e(x,\xi)_o)\end{align*}
and
\begin{align*}|d_\e(x, \xi)-d_\e(x, \xi')|\leq d_\e(\xi, \xi') \leq d_\e(x, \xi)+d_\e(x, \xi')\end{align*}
for all $x\in X$ and $\xi,\xi'\in \bd X$.
\end{itemize}
\end{fact}

The small range for visual parameters is neither optimal, nor particularly natural, and we have to consider the entire `cone' of visual metrics on $\bd X$. Statements about visual metrics on $\bd X$ are sometimes proved by first dealing with visual parameters in the small range, and then extended by scaling (Fact~\ref{scaling lemma}).

\subsection{Geometric group actions}\label{gga} Underlying the following definition is the fundamental fact that hyperbolicity is invariant under quasi-isometries. 

\begin{defn}[Gromov]
A group $\G$ is \emph{hyperbolic} if it acts geometrically, that is to say isometrically, properly and cocompactly, on a hyperbolic space. \end{defn}

We refer to a space carrying a geometric action of a hyperbolic group $\G$ as a \emph{geometric model for $\G$}. For example, Cayley graphs of $\G$ with respect to various finite generating sets are geometric models for $\G$. There could be, however, more natural geometric models for a given hyperbolic group, e.g., for a surface group of genus at least $2$ the natural geometric model is the standard hyperbolic plane. Geometric models for $\G$ have $\G$-equivariantly homeomorphic boundaries, and each one of them is a topological realization of $\bd \G$. 

Now let $X$ be a hyperbolic space admitting a geometric action of a (hyperbolic) group $\G$. In what follows, we assume that $\G$ is \emph{non-elementary}, that is, $\G$ is neither finite, nor virtually infinite cyclic. In terms of the space $X$, the non-elementary hypothesis on $\G$ means that $\bd X$ is infinite as a set.

The action of $\G$ on $X$ extends to the boundary $\bd X$. The boundary action is a convergence action, in the sense of \S\ref{sec: C0 and conv}. We also have $(gx,gx')_o\geq (x,x')_o-d(o,go)$ for all $g\in \G$ and $x,x'\in X$, which implies that 
\begin{align}\label{boundary action is Lipschitz}
(g\xi,g\xi')_o\geq (\xi,\xi')_o-d(o,go)
\end{align} 
for all $g\in \G$ and $\xi,\xi'\in \bd X$. Therefore $\G$ acts by Lipschitz maps on $(\bd X, d_\e)$ for any choice of visual metric $d_\e$ on $\bd X$.

\begin{defn}
The \emph{exponent}, or the \emph{volume entropy} of $X$ is the finite positive number given by
\[e_{X}=\inf \Big\{s>0 \: :\: \sum_{g\in \G}\exp(-sd(o,go))<\infty\Big\}=\limsup_{R\to \infty} \frac{1}{R}\ln \big|\{g\in\G: d(o,go)\leq R\}\big|.\]
\end{defn}
The two formulas give two interpretations of the exponent, namely critical exponent as well as growth exponent. As the notation $e_X$ already suggests, the definition  is independent of the basepoint $o\in X$ and of the group $\G$ acting geometrically on $X$.

The Patterson - Sullivan theory developed by Coornaert in \cite{Coo} plays a crucial role in understanding the growth of $\G$-orbits in $X$, and the Hausdorff dimensions and measures associated to visual metrics on $\bd X$. The following hold.

\begin{fact}[Orbit growth]\label{growth} 
Let $o$ be a basepoint in $X$. Then $\big|\{g\in\G: d(o,go)\leq R\}\big|\asymp \exp(e_{X} R)$.
\end{fact}

\begin{fact}[Hausdorff measure and dimension]\label{Ahlfors} 
Let $d_\e$ be a visual metric on $\bd X$. Then:

\begin{itemize}[leftmargin=20pt, itemsep=3pt]
\item the Hausdorff dimension $\hdim(\bd X, d_\e)$ equals $e_{X}/\e$;
\item the Hausdorff measure $\mu_\e$ is Ahlfors regular, that is $\mu_\e(B_r)\asymp r^{\hdim(\bd X, d_\e)}$
uniformly over all closed balls $B_r$ of radius $0\leq r\leq\mathrm{diam} (\bd X, d_\e)$.
\end{itemize}
\end{fact}

Fact~\ref{growth} is \cite[Thm.7.2]{Coo}. For sufficiently small visual parameters, Fact~\ref{Ahlfors} follows from \cite[Prop.7.4, Cor.7.5, Cor.7.6]{Coo}; using Fact~\ref{scaling lemma}, it extends to arbitrary visual parameters. 

In particular, $\bd X$ has finite mass under the Hausdorff measure defined by a visual metric. This allows for the following definition. 

\begin{defn} A \emph{visual probability measure} on $\bd X$ is a Borel probability measure $\mu$ satisfying $\mu \asymp \mu_\e$ for some (equivalently, each) Hausdorff measure $\mu_\e$ defined by a visual metric $d_\e$.
\end{defn}

\subsection{Visual dimension}\label{visual section} 
In general, there is no canonical choice of metric on the boundary $\bd X$ of a hyperbolic space $X$. So the notion of Hausdorff dimension for the boundary has to be understood relative to a family of admissible metrics. For visual metrics, one obtains the following notion of dimension.

\begin{defn}\label{defn: visual dim}
The \emph{visual dimension} of $\bd X$, denoted $\visdim \bd X$, is the infimal Hausdorff dimension of $(\bd X, d)$ as $d$ runs over the visual metrics on $\bd X$.
\end{defn}

Clearly $\visdim \bd X\geq \topdim \bd X$, as the topological dimension of a metric space is at most the Hausdorff dimension with respect to a compatible metric. A more appropriate comparison term for the visual dimension is that of conformal dimension. The following chain of inequalities holds:
\begin{align*}
\visdim \bd X\geq \textrm{A-confdim}\: \bd X\geq \textrm{confdim}\: \bd X\geq \topdim \bd X.
\end{align*}
The \emph{conformal dimension} of $\bd X$ is a notion of metric dimension which only depends on the quasi-isometry type of $X$. It resolves the metric ambiguity at the boundary by taking all possible metrics which are equivalent to a visual metric in a suitable sense. The original definition, due to Pansu, uses quasi-conformal equivalence; more recently, the closely related quasi-M\"obius equivalence seems to be favored. Then $\textrm{confdim}\: \bd X$ is defined as the infimal Hausdorff dimension of $(\bd X, d)$ as $d$ runs over all metrics  which are equivalent to a visual metric. See \cite[Section 14]{KB}. From a measure-theoretic point of view, the equivalence relation used for defining the conformal dimension is too loose. For hyperbolic spaces admitting geometric group actions, the notion of \emph{Ahlfors conformal dimension} strikes a compromise by restricting the equivalence relation to Ahlfors regular metrics. Namely, $\textrm{A-confdim}\: \bd X$ is defined as the infimal Hausdorff dimension of $(\bd X, d)$ as $d$ runs over all Ahlfors regular metrics on $\bd X$ which are quasi-M\"obius equivalent to a visual metric. The Ahlfors conformal dimension is a key concept for much of the current work on boundaries of hyperbolic spaces from the perspective of analysis on metric spaces \cite{Kle}.
 
We illustrate Definition~\ref{defn: visual dim} on the following important examples.

\begin{ex}\label{ex: trees}
Let $\mathcal{T}$ be a regular tree of degree at least $3$. Topologically, the boundary $\bd \mathcal{T}$ is a Cantor set. The Gromov product $(\cdot,\cdot)_o$ on $\mathcal{T}$ extends canonically and continuously to $\bd \mathcal{T}$, and $\exp(-(\cdot,\cdot)_o)$ is an ultrametric on $\bd \mathcal{T}$. Each $\e>0$ is a visual parameter, so $\visdim \bd\mathcal{T}=0$.
\end{ex}

\begin{ex}\label{ex: standard hyperbolic}
The boundary of the standard hyperbolic space $\mathbb{H}^{n}$, where $n\geq 2$, is the sphere $S^{n-1}$. The usual spherical metric is a visual metric with visual parameter $\e=1$. In fact $\e=1$ is the largest possible visual parameter. Indeed, the Lipschitz functions with respect to a visual metric are dense in $C(S^{n-1})$. On the other hand, they are the $\e$-H\"older functions with respect to the spherical metric, $\e$ being the visual parameter. Now observe that, on a geodesic metric space, only the constant functions are $\e$-H\"older for $\e> 1$. Thus $\visdim \bd\mathbb{H}^{n}=n-1$, the Hausdorff dimension with respect to the spherical metric.

More generally, let us consider the (non-compact) rank-$1$ symmetric spaces. These are the real, complex, quaternionic, or octonionic hyperbolic spaces $\mathbb{H}^n_K$, where $n\geq 2$ respectively $n=2$ in the exceptional octonionic case. Put $k=\dim_\R K\in \{1,2,4,8\}$. Topologically, the boundary $\bd \mathbb{H}^n_K$ is a sphere of dimension $nk-1$. The standard metric on $\bd \mathbb{H}^n_K$, the so-called Carnot metric, is a visual metric with visual parameter $\e=1$. As the Carnot metric is geodesic, no parameter greater than $1$ is a visual parameter. Therefore $\visdim \bd\mathbb{H}^n_K=\hdim\, \bd \mathbb{H}^n_K$, the Hausdorff dimension of $\bd \mathbb{H}^n_K$ equipped with the Carnot metric. The latter is explicitly given by the Mitchell - Pansu formula: $\hdim\, \bd \mathbb{H}^n_K=\topdim \bd \mathbb{H}^n_K+k-1\; (=nk+k-2)$.
\end{ex}

The notion of visual dimension for a hyperbolic space can be taken a step further. Namely, one could define a visual dimension for a hyperbolic group $\G$ as the infimal visual dimension of $\bd X$, where $X$ runs over the (isometry classes of) geometric models for $\G$. Our results are, in fact, most conveniently expressed in terms of such a visual dimension for the group.

%%%%%%%%%%%%%%%%%%%%%%%%%%%%%%%%%%%%%%%%%%%%%%%%%%%%%%%%%%%%%%
\section{The basic K-cycle for boundary actions of hyperbolic groups}\label{sec: finite summability}
In this section, we realize the paradigm of Section \ref{paradigm} in the case of a  hyperbolic group acting on its boundary. First, some standing notations for the rest of the paper. The main characters are
\begin{itemize}
\item[$\G$ :] a non-elementary hyperbolic group;
\item[$X$ :] a hyperbolic space on which $\G$ acts geometrically, i.e., a geometric model for $\G$;
\item[$\mu$ :] a visual probability measure on $\bd X$.
\end{itemize}
In order to be able to do geometric analysis on the boundary, we also fix
\begin{itemize}
\item[$d_\e$ :] a visual metric on $\bd X$,
\end{itemize}
and we denote
\begin{itemize}
\item[$D_\e$ :] the Hausdorff dimension of $(\bd X,d_\e)$.
\end{itemize}

\begin{lem}\label{double integral} 
Let $o\in X$ be a basepoint. Then there exists $C'>0$ such that, for all $g\in\G$, we have
\begin{align*}
\bigg(\iint d_\e(g\xi,g\xi')^{2}\: d\mu(\xi) \: d\mu(\xi')\bigg)^{1/2}\geq C' \exp(- \e\: d(o,go)).
\end{align*}
If $D_\e>2$, then there exists $C>0$ such that, for all $g\in\G$, we have
\begin{align*}
\bigg(\iint d_\e(g\xi,g\xi')^{2}\: d\mu(\xi) \: d\mu(\xi')\bigg)^{1/2}\leq C \exp(- \e\: d(o,go)).
\end{align*}
\end{lem}

The proof is deferred to the end of the section. The important part is the second estimate; the purpose of the first part is to show that we are getting the correct asymptotics. 

\begin{thm}\label{sharp general}
$(\lambda_{\mu}, P_{\ell^2\G})$ is a Fredholm module for $\hiX$ which is $D_\e^+$-summable when $D_\e>2$, respectively $p$-summable for every $p>2$ when $D_\e\leq 2$. The summability holds over the dense $*$-subalgebra $\Lip(\bd X, d_{\e})\ralg \G$, where $\Lip(\bd X, d_\e)$ is the $\G$-invariant algebra of Lipschitz functions on $\bd X$.
\end{thm}

\begin{proof} 
We first prove the claim in the case when $D_\e>2$. Using formula \eqref{iint}, we have 
\begin{align*}\dev\phi(g)\leq\|\phi\|_{\Lip}\;\sqrt{\tfrac{1}{2}\iint d_\e(g\xi,g\xi')^{2} \: d\mu(\xi) \: d\mu(\xi')}
\end{align*}
for all $\phi\in \Lip(\bd X, d_\e)$. It follows from Lemma~\ref{double integral} that there exists $C>0$ such that 
\begin{align*}
\dev\phi(g)\leq C\|\phi\|_{\Lip} \exp(-\e\: d(o,go))
\end{align*}
for all $\phi\in \Lip(\bd X, d_\e)$ and $g\in \G$. Let $T$ denote the multiplication by $g\mapsto \exp(-\e\: d(o,go))$, viewed as an operator on $\ell^2\G$. We claim that $T\in \mathcal{L}^{D_\e+} (\ell^2\G)$. Once we know this, it follows that multiplication by $\sigma\phi$ is in $\mathcal{L}^{D_\e+} (\ell^2\G)$ for all $\phi\in\Lip(\bd X, d_\e)$, and the proof of Proposition~\ref{from deviation to Fredholm} shows that $(\lambda_{\mu_\e}, P_{\ell^2\G})$ is a $D_\e^+$-summable Fredholm module. In order to prove our claim that $s_n(T)=O(n^{-1/D_\e})$, we first control a subsequence of singular values for $T$. Let $m_k$ denote the size of the ``ball'' $\{g\in \G: d(o,go)\leq k\}$; thus $m_k\asymp \exp(e_{X} k)$ by Fact~\ref{growth}. We have
\begin{align*}
s_{m_k+1}(T) < \exp(-\e k)= \exp(e_{X} k)^{-1/D_\e}\leq C_1\: m_k^{-1/D_\e}.
\end{align*}
For an arbitrary positive integer $n$, let $k$ be such that $m_k+1\leq n\leq m_{k+1}+1$. Then
\begin{align*}
s_{n}(T)\leq s_{m_k+1}(T)\leq C_1\: m_k^{-1/D_\e}\leq C_1\: n^{-1/D_\e} \Big(\frac{m_{k+1}+1}{m_k}\Big)^{1/D_\e}\leq C_2\: n^{-1/D_\e}
\end{align*}
for some constant $C_2$ independent of $n$ and $k$. The claim is proved for $D_\e>2$. 

Now assume that $D_\e\leq 2$, and let $p>2$. Let also $\alpha\in (0,1]$ so that $D_\e/\alpha$, which is the Hausdorff dimension of $(\bd X, d_{\e}^\alpha)$, satisfies $p>D_\e/\alpha>2$. By the previous part of the proof, we know that $(\lambda_{\mu}, P_{\ell^2\G})$ is $(D_\e/\alpha)^+$-summable over $\Lip(\bd X, d_{\e}^\alpha)\ralg \G$. As $\Lip(\bd X, d_{\e})$ is contained in $\Lip(\bd X, d_{\e}^\alpha)$, we conclude that $(\lambda_{\mu}, P_{\ell^2\G})$ is $p$-summable over $\Lip(\bd X, d_{\e})\ralg \G$. \end{proof}

In Theorem~\ref{sharp general}, the summability of the basic K-cycle $(\lambda_{\mu}, P_{\ell^2\G})$ always occurs above $2$. We do not know whether this phenomenon is due to some structural obstruction. There is, however, an obstruction to our method of controlling summability by the decay of the $\G$-deviation, and this is the fact that visual probability measures are doubly ergodic \cite{Kai2}. Indeed, by Proposition~\ref{2ergo} the $\G$-deviation has to decay faster than $\ell^2$ in the presence of double ergodicity.

We may optimize Theorem~\ref{sharp general} by varying the visual metric. Our notion of visual dimension is in fact tailored for this very purpose.

\begin{cor}\label{sharp corollary}
$(\lambda_{\mu}, P_{\ell^2\G})$ is $p$-summable for every $p>\max\{2,\visdim \bd X\}$. Furthermore, $(\lambda_{\mu}, P_{\ell^2\G})$ is $(\visdim \bd X)^+$-summable provided $\visdim \bd X$ is attained and greater than $2$.
\end{cor}

Now let us return to Lemma~\ref{double integral}, whose proof we have postponed.

\begin{proof}[Proof of Lemma~\ref{double integral}] The first estimate is straightforward. As in \eqref{boundary action is Lipschitz}, we have $(g\xi,g\xi')_o\leq d(o,go)+(\xi,\xi')_o$ for all $\xi,\xi'\in\bd X$. Therefore $d_\e(g\xi,g\xi')\geq c_1\exp(-\e\: d(o,go))\:d_\e(\xi,\xi')$ for some $c_1>0$, which implies that
\begin{align*}\bigg(\iint d_\e(g\xi,g\xi')^{2}\: d\mu(\xi) \: d\mu(\xi')\bigg)^{1/2} \geq c_2\exp(- \e\: d(o,go)).
\end{align*}
The second, conditional estimate is more involved. First, we assume that the visual parameter $\e$ is in the small visual range and that $d_\e$ is a visual metric enjoying the properties listed in Fact~\ref{small visual range}. We let $\alpha>0$, and we show the following: if $D_\e>2\alpha$, then there exists $C>0$ such that for all $g\in\G$ we have
\begin{align*}\bigg(\iint d_\e(g\xi,g\xi')^{2\alpha}\: d\mu(\xi) \: d\mu(\xi')\bigg)^{1/2}\leq C\exp(- \alpha\e\: d(o,go)).
\end{align*}
Let $\xi, \xi'\in\bd X$. Observe that $(gx,gx')_o+(g^{-1}o,x)_o+(g^{-1}o,x')_o\geq d(o,go)$ whenever $x,x'\in X$; indeed, this is just a complicated rewriting of the triangle inequality $d(o,x)+d(o,x')\geq d(x,x')$. Letting $x\to\xi, x\to\xi'$ and using \eqref{liminf limsup} and \eqref{liminf limsup 2}, we obtain that
\begin{align*}
(g\xi,g\xi')_o+(g^{-1}o,\xi)_o+(g^{-1}o,\xi')_o\geq d(o,go)-4\delta.
\end{align*}
Thus there is $C_1>0$ such that
\begin{align*}
d_\e(g\xi,g\xi')\leq C_1 \exp(-\e\: d(o,go))\: d_\e(g^{-1}o,\xi)^{-1}\: d_\e(g^{-1}o,\xi')^{-1}
\end{align*}
for all $g\in \G$ and $\xi,\xi'\in \bd X$. It follows that
\begin{align*}
\bigg(\iint d_\e(g\xi,g\xi')^{2\alpha}\: d\mu(\xi) \: d\mu(\xi')\bigg)^{1/2} \leq C_1^\alpha \exp(-\alpha\e\: d(o,go))\: \int d_\e(g^{-1}o,\xi)^{-2\alpha}\: d\mu(\xi)
\end{align*}
and our next goal is to show that the integral on the right hand side is bounded above independently of $g\in \G$. At this point, the technical advantage of using $\e$ in the small visual range becomes apparent: the function $d_\e(g^{-1}o,\cdot)$ is Lipschitz, in particular measurable, on $\bd X$. For each positive integer $k$ we put
\begin{align*}
\Delta_k= \big\{\xi\in \bd X: \exp(-\e k)\leq d_\e(g^{-1}o,\xi)\leq \exp(-\e (k-1))\big\}.
\end{align*}
(Although we will not need this fact, we remark that $d_\e(g^{-1}o,\cdot)$ is bounded below by a constant multiple of $\exp(-\e\: d(o,go))$, so $\Delta_k$ is in fact empty for $k\gg d(o,go)$.) From the hyperbolic inequality $(\xi,\xi')_o\geq \min\big\{(g^{-1}o,\xi)_o,(g^{-1}o,\xi')_o\big\}-\delta$, where $\xi, \xi'\in \bd X$, we deduce that there is $C_2\geq 0$ such that
\[d_\e(\xi,\xi')\leq C_2 \max\big\{d_\e(g^{-1}o,\xi),d_\e(g^{-1}o,\xi')\big\}\] 
for all $\xi, \xi'\in \bd X$. This inequality implies that $\mathrm{diam} (\Delta_k)\leq e^\e C_2\exp(-\e k)$. It now follows from Fact~\ref{Ahlfors} that there exists $C_3>0$, independent of $k$, such that
\begin{align*}
\mu(\Delta_k)\leq C_3\big(\exp(-\e k)\big)^{e_{X}/\e}= C_3 \exp(-e_{X}k).
\end{align*}
Using this diameter bound, we immediately get the desired integral estimate:
\begin{align*}
\int d_\e(g^{-1}o,\xi)^{-2\alpha} \: d\mu(\xi) &\leq \sum_{k\geq 1} \int_{\Delta_k} d_\e(g^{-1}o,\xi)^{-2\alpha} \: d\mu(\xi)\leq \sum_{k\geq 1} \exp(2\alpha\e k) \mu(\Delta_k)\\
&\leq C_3 \sum_{k\geq 1} \exp\big((2\alpha\e -e_X)k\big) 
\end{align*}
and the latter series converges since $D_\e=e_X/\e>2\alpha$ by assumption.

Now let $\e$ be an arbitrary visual parameter. Pick $\e_0$ in the small visual range, and let $d_{\e_0}$ be a corresponding visual metric. By Fact~\ref{scaling lemma}, we have
\begin{align*}
\bigg(\iint d_\e(g\xi,g\xi')^{2}\: d\mu(\xi) \: d\mu(\xi')\bigg)^{1/2}\asymp \bigg(\iint d_{\e_0}(g\xi,g\xi')^{2\e/\e_0}\: d\mu(\xi) \: d\mu(\xi')\bigg)^{1/2}.
\end{align*}
According to the lemma's hypothesis, $\hdim (\bd X,d_{\e_0})=(\e/\e_0)\:\hdim (\bd X,d_{\e})$ is greater than $2\e/\e_0$. Hence the previous part of the proof shows that the right-hand side is bounded above by a constant multiple of $\exp(-(\e/\e_0)\e_0\: d(o,go))=\exp(-\e\: d(o,go))$. \end{proof}

For the sake of conciseness, we adopt the following for the rest of the paper.

\begin{notation}
We write \emph{$D_\e^>$-summable} to mean 
\begin{align*}
\begin{cases}
D_\e^+\textrm{-summable} & \textrm{when } D_\e>2,\\
p\textrm{-summable for every } p>2 & \textrm{when } D_\e\leq 2.
\end{cases}
\end{align*}
\end{notation}

%%%%%%%%%%%%%%%%%%%%%%%
%%%%%%%%%%%%%%%%%%%%%%%%%%%%%%%%%%%%%%%%%%%%%%

\section{The boundary extension}\label{what is the boundary extension} 
All the basic K-cycles coming from visual probability measures on $\bd X$ are K-homologous, by Proposition~\ref{independence of comparable measures}. The purpose of this section is to describe the $\K^1$-class of a basic K-cycle as the class of a certain canonical extension of $\hiX$ by the compact operators on $\ell^2\G$. This extension encodes the compactification of $\G$ obtained by attaching the boundary $\bd X$. 

Let $\overline{\G}=\G\cup \bd X$ be the compactification of $\G$ by the boundary of the geometric model $X$. By definition, $g\to \omega\in \bd X$ in $\overline \G$ if $go\to \omega$ in $\overline{X}$ for some (equivalently, each) base point $o\in X$. From the exact sequence of $\G$-C*-algebras $0\to C_0(\G)\to C(\overline{\G})\to C(\bd X)\to 0$ we obtain an exact sequence of C*-crossed products by $\G$:
\begin{align}
\label{eq:crossed_extension}
0\Too C_0(\G)\rtimes \G\Too C(\overline{\G})\rtimes\G\Too C(\bd X)\rtimes\G\Too 0
\end{align}
Each C*-algebra in \eqref{eq:crossed_extension} is nuclear; in particular, the full and the reduced crossed products agree. The faithful representation of $C(\overline{\G})$ on $\ell^2\G$ by multiplication induces a faithful representation $\pi: C(\overline{\G})\rtimes\G\to \B(\ell^2\G)$, which restricts to the standard isomorphism between the ideal term $C_0(\G)\rtimes \G$ and the compact operators $\Comp (\ell^2\G)$. Thus \eqref{eq:crossed_extension} defines a class in the Brown - Douglas - Fillmore group $\mathrm{Ext}(\hi)$. The nuclearity of $\hi$ implies that $\mathrm{Ext}(\hi)$ and $\K^1(\hi)$ are isomorphic. The map $\mathrm{Ext} \to \K^1$ is given by the Stinespring construction, which dilates a completely positive map to an odd Fredholm module.

The compactification of $\G$, hence the exact sequence \eqref{eq:crossed_extension} as well, is defined in reference to the chosen geometric model $X$. However, and this is an important point, the boundaries of two geometric models for $\G$ are $\G$-equivariantly homeomorphic. It follows that, up to isomorphism of extensions, the exact sequence \eqref{eq:crossed_extension} is independent of the chosen geometric model $X$.

 \begin{defn}
\label{def:boundary_class}
The \emph{boundary extension class} $[\bd_\Gamma] \in \mathrm{Ext}(\hi)$ is the class defined by the extension \eqref{eq:crossed_extension}. 
\end{defn}

We will show that the Fredholm module $(\lambda_{\mu}, P_{\ell^2\G})$ represents $[\bd_\Gamma]$. The initial ingredient is the following lemma, which should be compared with Proposition~\ref{convergence action} and its proof. 

\begin{lem}\label{sharp delta} If $g\to \omega\in \bd X$ in $\overline \G$, then  $g_*\mu\to \delta_\omega$ in $\mathrm{Prob}(\bd X)$.
\end{lem}

\begin{proof} Fix $\phi \in C(\bd X)$. We have
\begin{align*}
\bigg| \int \phi\: d g_*\mu-\phi(\omega)\bigg|=\bigg| \int \phi(g\xi)-\phi(\omega)\: d\mu(\xi)\bigg| \leq \int \big|\phi(g\xi)-\phi(\omega)\big|\: d\mu(\xi).
\end{align*}
and we show that the right-hand integral converges to $0$ as $g\to \omega$ in $\overline \G$. Let $t>0$, and let $d_\e$ be a visual metric on $\bd X$ with parameter $\e$ in the small visual range so that Fact~\ref{small visual range} applies. The uniform continuity of $\phi$ provides us with some $R>0$ such that $|\phi(\xi)-\phi(\xi')|<t$ whenever $d_\e(\xi,\xi')<R$. The set
\begin{align*}
Z(g)=\{\xi\in \bd X: d_\e(g\xi,\omega)\geq R\}
\end{align*}
is measurable, since $\xi\mapsto d_\e(g\xi,\omega)$ is continuous. We write:
\begin{align*}
\int \big|\phi(g\xi)-\phi(\omega)\big|\: d\mu(\xi) &= \int_{\bd X\setminus Z(g)} \big|\phi(g\xi)-\phi(\omega)\big|\: d\mu(\xi) + \int_{Z(g)} \big|\phi(g\xi)-\phi(\omega)\big|\: d\mu(\xi)\\
&\leq  t+2\|\phi\|_\infty\: \mu (Z(g))
\end{align*}
It suffices to show that $\mu (Z(g))\to 0$ as $g\to \omega$. Let $o\in X$ be a basepoint. One easily checks that $(gx,w)_o+(g^{-1}o,x)_o\geq (go,w)_o$ for $x,w\in X$. Letting $x\to \xi$ and $w\to \omega$, and using \eqref{liminf limsup} and \eqref{liminf limsup 2}, we get $(g\xi,\omega)_o+(g^{-1}o,\xi)_o\geq (go,\omega)_o-3\delta$. It follows that there is $C_1>0$ such that
\begin{align*}
d_\e(g\xi,\omega)\:d_\e(g^{-1}o,\xi)\leq C_1\:d_\e(go,\omega)
\end{align*}
for all $g\in \G$ and $\xi,\omega\in\bd X$. Now by hyperbolicity there exists $C_2>0$ such that $d_\e(\xi,\xi')\leq C_2 \max\big\{d_\e(g^{-1}o,\xi),d_\e(g^{-1}o,\xi')\big\}$ for all $\xi,\xi'\in Z(g)$. In turn, both $d_\e(g^{-1}o,\xi)$ and $d_\e(g^{-1}o,\xi')$ are at most $C_1R^{-1}\: d_\e(g_o,\omega)$ by the inequality above. Thus $\mathrm{diam}(Z(g))\leq C_3 d_\e(go,\omega)$. By Ahlfors regularity (Fact~\ref{Ahlfors}) $\mu (Z(g))\leq C_4\: d_\e(go,\omega)^{e_{X}/\e}$. Now if  $g\to \omega$ then $d_\e(go,\omega)\to 0$, so $\mu (Z(g))\to 0$ as desired.
\end{proof}

In terms of the $\G$-expectation, Lemma~\ref{sharp delta} can be interpreted as follows: if $g\to \omega\in \bd X$ in $\overline \G$, then $\E \phi(g)\to \phi(\omega)$ for all $\phi\in C(\bd X)$. This means that we may extend continuous maps on the boundary $\bd X$ to continuous maps on the compactification $\overline{\G}$ by gluing a map to its $\G$-expectation. Namely, for $\phi\in C(\bd X)$ we define $\Ebar \phi\in C(\overline{\G})$ by 
\begin{align*}
\Ebar\phi= 
\begin{cases} \phi & \text{on $\bd X$}
\\
\E\phi &\text{on $\G$}.
\end{cases}
\end{align*}
We view $\Ebar: C(\bd X)\to C(\overline{\G})$ as a $\G$-equivariant, completely positive section for the quotient map $C(\overline{\G})\onto C(\bd X)$ given by restriction. We immediately obtain a completely positive section for the quotient map $C(\overline{\G})\rtimes \G\onto \hiX$.

\begin{thm}\label{representing boundary class} Define 
\begin{align}\label{def sec}
s_\mu : C(\bd X)\ralg \G \to C(\overline{\G})\rtimes \G, \qquad s_\mu\big(\sum \phi_g\, g\big) = \sum \big(\Ebar \phi_g\big)\, g.
\end{align} 
Then $s_\mu$ extends to a completely positive section for the quotient map $C(\overline{\G})\rtimes \G\onto \hiX$. Consequently, $(\lambda_{\mu}, P_{\ell^2\G})$ is a Fredholm module representing the boundary extension class $[\bd_\G]$.
\end{thm}

\begin{proof} Recall that $\pi: C(\overline{\G})\rtimes\G\to \B(\ell^2\G)$ is the representation induced by the multiplication representation of $C(\overline{\G})$ on $\ell^2\G$. We claim that
\begin{align}\label{stined}
\pi s_\mu(a)= P_{\ell^2\G} \lambda_{\mu}(a) P_{\ell^2\G}
\end{align}
 for all $a\in C(\bd X)\rtimes _{\textrm{alg}} \G$. Indeed, for $\phi \in C(\bd X)$ and $g \in  \G$ we have
\begin{align*}\pi s_\mu(\phi g)=\pi\big(\Ebar\phi\big)\pi(g)= \big(P_{\ell^2\G} \lambda_{\mu}(\phi) P_{\ell^2\G}\big)\big(P_{\ell^2\G} \lambda_{\mu}(g) P_{\ell^2\G}\big)=P_{\ell^2\G} \lambda_{\mu}(\phi g) P_{\ell^2\G}
\end{align*}
by using \eqref{projection compressed lambda}, and the fact that $\lambda_{\mu}|_\G$ commutes with $P_{\ell^2\G}$. Since $\pi$ is faithful, and therefore isometric,  \eqref{stined} implies that $s_\mu$ extends by continuity to a completely positive section for the quotient map $C(\overline{\G})\rtimes \G\onto \hiX$. The Stinespring dilation $\pi s_\mu= P_{\ell^2\G}\: \lambda_{\mu}\: P_{\ell^2\G}$ precisely means that $(\lambda_{\mu}, P_{\ell^2\G})$ represents $[\bd_\G]$. \end{proof}
 
Theorem~\ref{representing boundary class} and Proposition ~\ref{finite index} imply the following. 

\begin{prop}\label{finite index boundary} Let $\Lambda$ be a finite-index subgroup of $\G$. Then the restriction homomorphism $\mathrm{res}: \K^1(\hi)\to\K^1(C(\bd \Lambda)\rtimes \Lambda)$  sends $[\bd_\G]$ to $[\G:\Lambda] \cdot [\bd_{\Lambda}]$.
\end{prop}

Indeed, the comparability condition required in the statement of Proposition ~\ref{finite index} is satisfied in our setting. Since $\G$ acts by Lipschitz maps on $\bd X$ for any choice of visual metric $d_\e$, the corresponding Hausdorff measure satisfies $g_*\mu_\e\asymp\mu_\e$ for all $g\in \G$; the same will then hold for any visual probability measure on $\bd X$.

%%%%%%%%%%%%%%%%%%%%%%%%%%%%%%%%%%%%
%%%%%%%%%%%%%%%%%%%%%%%%%%%%%%%%%%%%

\section{Poincar\'e duality and twisted K-cycles}
\subsection{Poincar\'e duality} Poincar\'e duality for $\hi$, proved in \cite{Emerson}, plays an essential role in this paper. The proof from \cite{Emerson}, though most likely not Poincar\'e duality itself, needs the following mild symmetry condition on the boundary.

\begin{defn}\label{defn-regular}
A hyperbolic group is \emph{regular} if  its boundary admits a continuous self-map without fixed-points.
\end{defn}

Regularity in the above sense is satisfied whenever the boundary is a topological sphere, a Cantor set or a Menger compactum. We are not aware of any example where regularity fails. The `topologically rigid' hyperbolic groups of Kapovich and Kleiner \cite{KK} come close, though, for their boundaries admit no self-homeomorphisms without fixed points (but regularity does not require a homeomorphism, merely a 
map).  

Poincar\'e duality is defined by a cup-cap procedure explained in \cite{Emerson}. A C*-algebra is \emph{Poincar\'e self-dual} in this sense if there are two 
`fundamental classes', 
\begin{align*} 
\Delta \in \KK_1(A\otimes A, \C), \qquad \widehat{\Delta}\in \KK_1(\C, A\otimes A)
\end{align*}
 satisfying certain equations which we do not specify here (the zig-zag equations of the theory of adjoint functors.) Given \(\Delta\) as above, a `Poincar\'e duality' 
 map is defined by 
 \begin{align*} 
 \Delta\cap \colon \K_*(A) \to \K^{*+1}(A),\qquad  \Delta \cap x = (x\otimes_\C 1_A)\otimes_{A\otimes A} \Delta,
 \end{align*}
 or in other words, by composing in \(\KK\), the morphisms \(x\otimes 1_A\in \KK_*(A, A\otimes A)\) with \(\Delta\in \KK_1(A\otimes A, \C)\), to get a morphism in \(\KK_{*+1}(A, \C) = \K^{*+1}(A)\). 
 
For $\hi$, such fundamental classes $\Delta$ and $\widehat{\Delta}$ are constructed in \cite{Emerson}, and, using the Baum - Connes machinery, it is shown that \(\Delta \cap \) is an isomorphism when $\G$ is regular and torsion-free. The inverse isomorphism comes from \(\widehat{\Delta}\). The class \(\Delta\) is defined as an \emph{extension}, \emph{i.e.} as a cycle for the Brown - Douglas - Fillmore group \(\mathrm{Ext}(\hi\,\otimes \,\hi, \C)\), whose Busby invariant is as follows. First let 
\begin{align*}
\tau\colon \hi \to \Calk
\end{align*}
be the Busby invariant of the extension \eqref{eq:crossed_extension}. Thus
\(\tau\) is the integrated form of the covariant pair associated to the regular representation \(\lambda \colon \G \to \B(\ell^2\G)\) followed by the quotient map 
\(\B \to \B/\Comp\), and the map \(\phi\mapsto \M (\tilde{\phi})\) followed by the 
quotient map, where \(\tilde{\phi}\) is an extension of \(\phi\in C(\bd \G)\) to a continuous function on \(\overline{\G}\). Let
\begin{align*}
\tauop\colon \hi \to \Calk, \qquad \tauop(a) := 
J\tau (a)J
\end{align*} 
where \(J\) is the symmetry 
\begin{align*}
J: \ell^2\G\to \ell^2\G,\qquad  J(\delta_g) := \delta_{g^{-1}}.
\end{align*} 
Note that \(\sop\colon \hi \to JC(\overline{\G})\rtimes \G J \subset \B(l^2\G)\), 
defined by \(\sop (a) := Js(a)J\) is a completely positive splitting of \(\tauop\).

As proved in \cite{Emerson}, \(\tau\) and \(\tauop\) commute so they 
combine to give a unital $*$-homomorphism \(\hi\,\otimes \,\hi\to 
\Calk\), \(a\otimes b \mapsto \tau (a)\tauop (b)\). The class \(\Delta\) is 
by definition the pre-image of this extension under the canonical map 
\begin{align*} 
 \KK_1(\hi\,\otimes \,\hi, \C) \to  \mathrm{Ext}(\hi\, \otimes \, \hi, \C),
 \end{align*}
which is an isomorphism since \(\hi\,\otimes \,\hi\) is nuclear. However, a \emph{cycle} in \(\KK_1\) representing \(\Delta\) is difficult to describe because we 
do not know of a concrete completely positive splitting of the extension 
defining \(\Delta\). 

Nevertheless, we show that the ideas of the previous sections can be used 
to compute the Poincar\'e duality \emph{map} \(\Delta \cap\) in explicit terms.

\begin{lem}
\label{whatispd}
Let \(e\in \hi\) be a projection. Then the Poincar\'e dual \(\Delta \cap [e]\) of the class \([e]\in \K_0(\hi)\) is the class in \(\K^1(\hi)  = \KK_1(\hi, \C) \cong \mathrm{Ext}(\hi, \C)\) of the extension 
with Busby invariant 
\begin{align*}\tau_e: \hi \to \B(\ell^2\G)/\Comp(\ell^2\G), \qquad \tau_e(a):= \tauop(e) \tau (a)\tauop (e).
\end{align*}
Let \(u \in \hi\) be a unitary, and let
\begin{align*}
\bar{u}\colon C_0(\R) \subset C(S^1) \to \hi
\end{align*}
denote the composition of the usual inclusion of \(C_0(\R) \cong C_0(S^1-\{1\})\) into \(C(S^1)\), functional calculus for \(u\). Then the Poincar\'e dual 
\(\Delta \cap [u]\) of the class \([u]\in \K_1(\hi)\) is the class in \(\KK_1(C_0(\R)\otimes \hi, \C)\) of the extension of \(C_0(\R)\otimes \hi\) by \(\Comp(\ell^2\G)\) whose Busby invariant is 
\begin{align*} 
\tau_u(f\otimes a) = (\tauop\circ \bar{u})(f)\tau (a).
\end{align*}
\end{lem}

\begin{proof}
Both assertions follow from the functoriality of the Kasparov and Brown - Douglas - Fillmore theories. We prove the first assertion, concerning projections. The second one, concerning unitaries, is proved similarly. If \(e\) is a projection in \(\hi\), then \([e] = e_*([1])\) where \(e_*\colon \K_*(\C) \cong \Z \to \K_*(\hi)\) is the $*$-homomorphism induced by \(e\), and 
the element 
\begin{align*}
[e]\otimes 1_{\hi}\in \KK_0\big(\hi, \hi \,\otimes \,\hi\big)
\end{align*} 
is represented by the $*$-homomorphism \(e\otimes 1_{\hi}\), and composing with \(\Delta\) amounts to composing the Busby invariant for \(\Delta\) and the $*$-homomorphism \(e\otimes 1\). By the definitions and the fact that \(\tauop (e)\) is a projection in the Calkin algebra commuting with \(\tau (a)\) for all \(a\), this yields \(\tau_e\).  \end{proof}

Note that the boundary extension class $[\bd_\G]\in \K^1(\hi)$ is the Poincar\'e dual of the unit class $[1_{\hi}]\in \K_0(\hi)$. On the other hand, the K-theory Gysin sequence of \cite{Emerson:Euler} shows that the order of $[1_{\hi}]$ in $\K_0(\hi)$ is determined by the Euler characteristic of $\G$. Consequently: 

\begin{prop}[Emerson, Emerson - Meyer]\label{order tf} Assume that $\G$ is regular and torsion-free. Then $[\bd_\G]= 0$ in $\K^1(\hi)$ if and only if $\chi(\G)=\pm 1$. Furthermore, $[\bd_\G]$ has infinite order in $\K^1(\hi)$ if and only if $\chi(\G)=0$. 
\end{prop}

We may pass from torsion-free to virtually torsion-free groups, and establish a version of Proposition~\ref{order tf} for this much larger class, by using Proposition ~\ref{finite index boundary}. The \emph{rational Euler characteristic} of a virtually torsion-free group $\G$ is defined by the formula $\chi(\G)=\chi(\Lambda)/[\G: \Lambda]$ where $\Lambda$ is any torsion-free subgroup of finite index. 

\begin{cor}
Assume that $\G$ is regular and virtually torsion-free. If $\chi(\G)\notin 1/\Z$ then $[\bd_\G]\neq 0$ in $\K^1(\hi)$. If $\chi(\G)=0$ then $[\bd_\G]$ has infinite order in $\K^1(\hi)$.
\end{cor}

The assumption of virtual torsion-freeness is a very mild one. A long-standing open problem asks whether all hyperbolic groups are virtually torsion-free.

\subsection{Twisted K-cycles} We now show how to compute \(\Delta \cap \) in terms of certain canonical Fredholm modules obtained by `twisting' the basic K-cycle \((\lambda_\mu, \projcon)\). 

The \emph{right regular representation}  $\lambdaop_{\mu}$ of  $\hiX$ on $\ell^2(\G, L^2(\bd X,\mu))$ is given as follows:
\begin{align*}
\lambdaop_{\mu}(\phi)\Big(\sum \psi_h \delta_h\Big)=\sum (h.\phi)\psi_h \delta_h, \qquad \lambdaop_{\mu}(g)\Big(\sum \psi_h \delta_h\Big)=\sum \psi_h \delta_{hg^{-1}}
\end{align*}
for $\phi\in C(\bd X)$, $g\in \G$. Note the covariance relation $\lambdaop_\mu(g.\phi)=\lambdaop_\mu(g)\lambdaop_\mu(\phi)\lambdaop_\mu(g^{-1})$. The right and the left regular representations do not commute, but they do satisfy
\begin{align*}
[\lambda_\mu(\phi), \lambdaop_\mu(\phi')]=0, \qquad [\lambda_\mu(g), \lambdaop_\mu(g')]=0
\end{align*}
for all $\phi,\phi'\in C(\bd X)$ and $g,g'\in \G$.

The symmetry $J$ on $\ell^2\G$ has an obvious extension $J\otimes \mathrm{id}$ to $\ell^2(\G, L^2(\bd X,\mu))$, and, using 
the same notation for this extension, we have
\begin{align*}
\lambdaop_\mu=J\lambda_\mu J.
\end{align*}
Since \(s\) splits \(\tau\), the image in the Calkin algebra of 
\(\projcon \lambda_\mu (a) \projcon=s(a)\) is \(\tau (a)\). Also, 
\( \projcon \lambdamuop (a)\projcon = \sop(a)\), and its image in the Calkin algebra 
is \(\tauop (a)\). Hence 
\begin{equation}
\label{helly}
 \projcon \lambda_\mu (a)\projcon = \tau (a)  \; \textup{mod} \; \Comp, \;\; \textup{and} \; 
\projcon\lambdamuop (a)\projcon = \tauop (a) \; \textup{mod} \; \Comp.
\end{equation}

\begin{thm}\label{twisted representatives}
Let $\G$ be regular and torsion-free. Then the following hold.
\begin{itemize}[leftmargin=20pt, itemsep=3pt]
\item Every class in $\K^1(\hi)$ is represented by an odd Fredholm module of the form 
\begin{align*}
\big(\lambda_\mu, P_{\ell^2\G}\lambdaop_\mu(e)P_{\ell^2\G}\big), \qquad e \textrm{ projection in }\hi.
\end{align*}
\item Every class in $\K^0(\hi)$ is represented by a balanced even Fredholm module of the form
\begin{align*}
\big(\lambda_\mu, P_{\ell^2\G}\lambdaop_\mu(u)P_{\ell^2\G}+(1-P_{\ell^2\G})\big), \qquad u \textrm{ unitary in } \hi.
\end{align*}
\end{itemize}
\end{thm}

\begin{proof}
Let \(Q_e:= \projcon \lambdamuop (e)\projcon\). 
By the discussion preceding the Theorem, \(Q_e = \tauop (e)\) mod compact operators. Hence \(Q_e\) is a self-adjoint, essential 
projection. Furthermore, the commutator $[\lambda_\mu (a), Q_e]$ is compact for all \(a\in \hi\). For mod compacts we have
\begin{align*} 
\lambda_\mu (a)Q_e & = \lambda_\mu (a) \projcon \lambdamuop (e) \projcon = 
\projcon \lambda_\mu (a) \projcon\lambdamuop (e)\projcon\\
& = \tau (a) \tauop (e) = \tauop (e)\tau (a) = Q_e \lambda_\mu (a).
\end{align*} 
This shows that \( (\lambda_\mu, P_{\ell^2\G}\lambdaop_\mu(e)P_{\ell^2\G})\) is a Fredholm module. The map \(\KK_1\) to \(\mathrm{Ext}\) sends its class to the extension with Busby invariant 
\(a\mapsto Q_e \lambda_\mu (a)Q_e \; \textup{mod} \; \Comp\) and this equals 
\(\tau_e (a):= \tauop (e) \tau (a)\tauop (e)\). Hence \((\lambda_\mu, P_{\ell^2\G}\lambdaop_\mu(e)P_{\ell^2\G})\) represents \([\tau_e (a)]\), which equals \(\Delta \cap [e]\) by Lemma \ref{whatispd}. 

The second assertion is proved by combining the same observations with Bott periodicity, see Lemma 2 of \cite{Emerson}. 
\end{proof}

We are implicitly using the fact that all the K-theory classes of a unital, simple and purely infinite C*-algebra can be represented by non-zero projections, respectively unitaries of the C*-algebra \cite{Cun}. 

%%%%%%%%%%%%%%%%%%%%%%%%%%%%%%%%%%%%
%%%%%%%%%%%%%%%%%%%%%%%%%%%%%%%%%%%%

\section{Summability of the twisted K-cycles}
In this section we prove the following theorem.

\begin{thm}\label{summability achieved}
Let $\G$ be regular and torsion-free. Then $C(\bd X)\rtimes \G$ has uniformly $D_\e^>$-summable K-homology over $\Lip(\bd X,d_\e)\ralg \G$.
\end{thm}

Theorem~\ref{intro: uniform for boundary} is an immediate corollary. In turn, Theorem~\ref{intro: uniform for boundary} and Examples~\ref{ex: trees}, ~\ref{ex: standard hyperbolic} yield the following explicit applications.

\begin{cor} If $\G$ is a finitely generated free group, then the K-homology of $\hi$ is uniformly $p$-summable for every $p>2$. If $\G$ is a torsion-free cocompact lattice in $\SO(n,1)$, then the K-homology of $C(S^{n-1})\rtimes \G$ is uniformly $(n-1)^+$-summable when $n\geq 4$, respectively uniformly $p$-summable for every $p>2$ when $n=2,3$. If $\G$ is a torsion-free cocompact lattice in $\SU(n,1)$, then the K-homology of $C(S^{2n-1})\rtimes \G$ is uniformly $(2n)^+$-summable. If $\G$ is a torsion-free cocompact lattice in $\Sp(n,1)$, then the K-homology of $C(S^{4n-1})\rtimes \G$ is uniformly $(4n+2)^+$-summable.
\end{cor}

To prove Theorem~\ref{summability achieved}, we start with an integral estimate.

\begin{lem}\label{single integral} 
Let $o\in X$ be a basepoint, and assume that $D_\e>2$. Then there exists $C>0$ such that, for all $g, h\in\G$, we have
\begin{align*}
\int d_\e(g\xi, h\xi)\: d\mu(\xi) \leq C \exp(-\e\: (go,ho)_o).
\end{align*}
\end{lem}

\begin{proof} 
The proof  is similar to that of Lemma~\ref{double integral}. First, we assume that the parameter $\e$ is in the small visual range and we let $\alpha>0$. We claim the following: if $D_\e>2\alpha$, then there exists $C>0$ such that for all $g, h\in\G$ we have
\begin{align}\label{a temporary claim}
\int d_\e(g\xi, h\xi)^{\alpha}\: d\mu(\xi) \leq C \exp(-\alpha\e\: (go,ho)_o).
\end{align}
Pick $\xi\in\bd X$. Observe that $(gx,hx)_o+(g^{-1}o, x)_o+(h^{-1}o, x)_o\geq  (go,ho)_o$ for $x\in X$; indeed, this amounts to $d(gx,hx)-d(go,ho)\leq 2d(o,x)$. Letting $x\to \xi$ and using \eqref{liminf limsup} and \eqref{liminf limsup 2}, we get
\begin{align*}
(g\xi,h\xi)_o+(g^{-1}o,\xi)_o+(h^{-1}o,\xi)_o\geq (go,ho)_o-4\delta.
\end{align*}
Hence, there is $C_1\geq 0$ such that, for all $\xi\in\bd X$, we have 
\begin{align}\label{a temporary distance bound}
d_\e(g\xi,h\xi)\leq C_1 \exp(-\e\: (go,ho)_o)\: d_\e(g^{-1}o,\xi)^{-1}\: d_\e(h^{-1}o,\xi)^{-1}.
\end{align}
Recall from the proof of Lemma~\ref{double integral} that
\begin{align*}
\int d_\e(g^{-1}o,\xi)^{-2\alpha} \: d\mu(\xi) \leq C_2 \end{align*}
independently of $g\in \G$. By the Cauchy-Schwartz inequality, it follows that 
\begin{align}\label{a cauchy-schwartz inequality}
\int d_\e(g^{-1}o,\xi)^{-\alpha}\: d_\e(h^{-1}o,\xi)^{-\alpha} \: d\mu(\xi)\leq C_2 
\end{align}
independently of $g, h\in \G$. Now ~\eqref{a temporary distance bound} and ~\eqref{a cauchy-schwartz inequality} yield ~\eqref{a temporary claim}.

The remainder of the proof goes just like the last step in the proof of Lemma~\ref{double integral}. Let $\e$ be an arbitrary visual parameter, let $\e_0$ be in the small visual range, and let $d_{\e_0}$ be a visual metric as in Fact~\ref{small visual range}. We have
\begin{align*}\int d_\e(g\xi, h\xi)\: d\mu(\xi)\asymp \int d_{\e_0}(g\xi, h\xi)^{\e/\e_0}\: d\mu(\xi)
\end{align*}
by Fact~\ref{scaling lemma}. As $\hdim (\bd X,d_{\e_0})>2\e/\e_0$, the previous part says that a constant multiple of $\exp(-(\e/\e_0)\e_0\:  (go,ho)_o)=\exp(-\e\:  (go,ho)_o)$ is an upper bound for the right hand side. 
 \end{proof}

\begin{lem}
\label{prop:quantifiedcommutation}
For all $a,b \in \Lip(\bd X,d_\e)\ralg \G$, the commutator $\big[\lambda_\mu(a), P_{\ell^2\G}\lambdaop_\mu(b)P_{\ell^2\G}\big]$ is $D_\e^>$-summable.
\end{lem}

\begin{proof} Assume that $D_\e>2$. The case when $D_\e\leq 2$ is deduced as in the proof of Theorem~\ref{sharp general}.

We recall from Theorem~\ref{sharp general} that $\lambda_\mu(a)$ commutes mod $\mathcal{L}^{D_\e+}$ with $P_{\ell^2\G}$. It follows that 
\begin{align*}
\big[\lambda_\mu(a), P_{\ell^2\G}\lambdaop_\mu(b)P_{\ell^2\G}\big]=\big[P_{\ell^2\G}\lambda_\mu(a)P_{\ell^2\G}, P_{\ell^2\G}\lambdaop_\mu(b)P_{\ell^2\G}\big]\quad \textrm{ mod }\mathcal{L}^{D_\e+}
\end{align*}
and the right-hand commutator is, with our notations, $[s_\mu(a), \sop_\mu(b)]$. Clearly, the property that $[s_\mu(a), \sop_\mu(b)]\in \mathcal{L}^{D_\e+}$ is additive in $a$ and $b$. Now $s_\mu(aa')=s_\mu(a)s_\mu(a')$ mod $\mathcal{L}^{D_\e+}$ for $a,a'\in\Lip(\bd X,d_\e)\ralg \G$, hence $\sop_\mu=Js_\mu J$ is multiplicative mod $\mathcal{L}^{D_\e+}$ on $\Lip(\bd X,d_\e)\ralg \G$ as well. Therefore, the property that $[s_\mu(a), \sop_\mu(b)]\in \mathcal{L}^{D_\e+}$ is also multiplicative in $a$ and $b$. We thus see that it suffices to treat the case when $a$ and $b$ are either Lipschitz functions or group elements.

For all $g,g'\in \G$ and $\phi, \phi'\in \Lip(\bd X,d_\e)$ we have
\begin{align*}
[s(g),\sop(g')]=0, \qquad [s(\phi), \sop(\phi')]=0, \qquad [s(g), \sop (\phi)] = -J[s(\phi), \sop(g)]J.
\end{align*} It therefore suffices to analyze the summability of the commutator $[s(\phi), \sop(g)]$ or, more conveniently, the summability of $\sop(g^{-1})[s(\phi), \sop(g)]$ which is readily verified to be multiplication by $h\mapsto \E\phi(hg^{-1})- \E\phi(h)$ on $\ell^2\G$.

 For $h_1,h_2\in\G$ and $\phi\in \Lip(\bd X,d_\e)$ we have
\begin{align*}
\big|\E\phi(h_1)- \E\phi(h_2)\big| \leq \int \big|\phi(h_1\xi)-\phi(h_2\xi)\big|\: d\mu(\xi)  \leq \|\phi\|_{\Lip} \int d_\e(h_1\xi, h_2\xi)\: d\mu(\xi)
\end{align*}
so, by Lemma~\ref{single integral}, there exists a constant $C>0$ such that
\begin{align*}
\big|\E\phi(h_1)- \E\phi(h_2)\big| \leq C \|\phi\|_{\Lip} \exp(-\e\: (h_1o,h_2o)_o).
\end{align*}
Put $h_1=hg^{-1}$ and $h_2=h$. As $(hg^{-1}o,ho)_o=(g^{-1}o,o)_{h^{-1}o}\geq d(o,ho)-d(o,go)$, we obtain
\begin{align*}\label{towards lipschitzness II}
\big|\E\phi(hg^{-1})- \E\phi(h)\big| \leq C \|\phi\|_{\Lip} \exp(\e\: d(o,go))\exp(-\e\: d(o,ho)).
\end{align*}
Finally, we recall from the proof of Theorem~\ref{sharp general} that multiplication by $h\mapsto \exp(-\e\: d(o,ho))$ is in $\mathcal{L}^{D_\e+} (\ell^2\G)$.  \end{proof}

\begin{thm}\label{enough are summable}
There is a smooth subalgebra $\mathcal{A}\subseteq C(\bd X)\rtimes \G$, containing $\Lip(\bd X,d_\e)\ralg \G$, such that the odd, respectively even, Fredholm modules
\begin{align*}
&\big(\lambda_\mu, P_{\ell^2\G}\lambdaop_\mu(e)P_{\ell^2\G}\big), \qquad e \textrm{ projection in }\mathcal{A}\\
&\big(\lambda_\mu, P_{\ell^2\G}\lambdaop_\mu(u)P_{\ell^2\G}+(1-P_{\ell^2\G})\big), \qquad u\textrm{ unitary in } \mathcal{A}
\end{align*}
are $D_\e^>$-summable Fredholm modules over $\Lip(\bd X,d_\e)\ralg \G$.
\end{thm}

We recall that a subalgebra $\mathcal{A}$ of a C*-algebra $A$ is said to be smooth if it is dense and stable under holomorphic calculus. Then the projections of $\mathcal{A}$ are dense in the projections of $A$, and the unitaries of $\mathcal{A}$ are dense in the unitaries of $A$. In particular, the K-theory classes of $A$ can be represented by projections, respectively unitaries, from $\mathcal{A}$. In light of  this fact, Theorem~\ref{summability achieved} follows by combining Theorem~\ref{enough are summable} and Theorem~\ref{twisted representatives}.

 \begin{proof}[Proof of Theorem~\ref{summability achieved}] To fix ideas, let us assume that $D_\e>2$.
 
Consider the $*$-subalgebra
\begin{align*}
\mathcal{A}=\big\{a\in \hiX: \big[\lambda_\mu(a), P_{\ell^2\G}\lambdaop_\mu(b)P_{\ell^2\G}\big]\in \mathcal{L}^{D_\e+} \textrm{ for all } b\in \Lip(\bd X,d_\e)\ralg \G\big\}.
\end{align*}
Then $\mathcal{A}$ contains $\Lip(\bd X,d_\e)\ralg \G$, by Lemma~\ref{prop:quantifiedcommutation}, in particular $\mathcal{A}$ is dense in $\hiX$. To see that $\mathcal{A}$ is stable under holomorphic calculus, consider for a moment the subalgebra $\mathcal{A}_b=\{a\in \hiX: \big[\lambda_\mu(a), P_{\ell^2\G}\lambdaop_\mu(b)P_{\ell^2\G}\big]\in \mathcal{L}^{D_\e+}\}$ corresponding to a \emph{fixed} $b\in \Lip(\bd X,d_\e)\ralg \G$. Then $\mathcal{A}_b$ is stable under holomorphic calculus in $\hiX$, therefore the same holds true for $\mathcal{A}$, the intersection of the family of subalgebras $\{\mathcal{A}_b: b\in \Lip(\bd X,d_\e)\ralg \G\}$.
 
Next, we prove the summability claim. Note first that 
\begin{align*}
\big[\lambda_\mu(a), P_{\ell^2\G}\big]\in \mathcal{L}^{D_\e+}, \qquad \textrm{for all }a \in \mathcal{A}.
\end{align*}
Then also $\big[\lambdaop_\mu(a), P_{\ell^2\G}\big]\in \mathcal{L}^{D_\e+}$ for all $a \in \mathcal{A}$. It follows that $P_{\ell^2\G}\lambdaop_\mu(e)P_{\ell^2\G}$ is a projection mod $\mathcal{L}^{D_\e+}$ whenever $e\in \mathcal{A}$ is a projection, and that $P_{\ell^2\G}\lambdaop_\mu(u)P_{\ell^2\G}+(1-P_{\ell^2\G})$ is a unitary mod $\mathcal{L}^{D_\e+}$ whenever $u\in \mathcal{A}$ is a unitary.

By definition, if $a \in \mathcal{A}$ then the commutator $\big[\lambda_\mu(a), P_{\ell^2\G}\lambdaop_\mu(b)P_{\ell^2\G}\big]$ is $D_\e^+$-summable for every $b\in \Lip(\bd X,d_\e)\ralg \G$. As already hinted in the proof of Lemma~\ref{prop:quantifiedcommutation}, the summability of the above commutators is in fact symmetric in $a$ and $b$. Indeed, using the fact that $\big[\lambda_\mu(a), P_{\ell^2\G}\big]\in \mathcal{L}^{D_\e+}$ for all $a \in \mathcal{A}$, we may write\begin{align*}
\big[\lambda_\mu(a), P_{\ell^2\G}\lambdaop_\mu(b)P_{\ell^2\G}\big]&=\big[P_{\ell^2\G}\lambda_\mu(a)P_{\ell^2\G}, P_{\ell^2\G}\lambdaop_\mu(b)P_{\ell^2\G}\big]\quad \textrm{ mod }\mathcal{L}^{D_\e+},\\
\big[\lambda_\mu(b), P_{\ell^2\G}\lambdaop_\mu(a)P_{\ell^2\G}\big]&=\big[P_{\ell^2\G}\lambda_\mu(b)P_{\ell^2\G}, P_{\ell^2\G}\lambdaop_\mu(a)P_{\ell^2\G}\big]\quad \textrm{ mod }\mathcal{L}^{D_\e+}.
\end{align*}
Now observe that the right-hand side commutators are conjugate by the symmetry $J$. This shows that, if $b\in \Lip(\bd X,d_\e)\ralg \G$, then the commutator $\big[\lambda_\mu(b), P_{\ell^2\G}\lambdaop_\mu(a)P_{\ell^2\G}\big]$ is $D_\e^+$-summable for every $a \in \mathcal{A}$. We conclude that the indicated Fredholm modules are $D_\e^+$-summable over $\Lip(\bd X,d_\e)\ralg \G$.
\end{proof}

%%%%%%%%%%%%%%%%%%%%%%%%%%%%%%%%%%%%%%%%
%%%%%%%%%%%%%%%%%%%%%%%%%%%%%%%%%%%%%%%%
\section{The K-homology Gysin sequence for boundary actions }\label{gysin section}
In this section, we attack the problem of proving that the reduced C*-algebra of a hyperbolic group has uniformly summable \(\K\)-homology. This involves some tools from \(\KK\)-theory. We start by summarizing the basic facts we will need about `\(\gamma\)-elements' and the Dirac dual-Dirac method. 

\subsection{Descent, \(\gamma\)-elements}
For any discrete group (or more generally locally compact group) \(\G\), 
`descent,' in Kasparov theory, refers to a natural map 
\[ j\colon \KK^\G_*(A,B) \to \KK_*(A\rtimes \G, B\rtimes \G)\]
which extends to equivariant \(\KK\)-cycles (and homotopies) the process of integrating a 
\(\G\)-equivariant $*$-homomorphism \(A\to B\) to an ordinary $*$-homomorphism 
\(A\rtimes \G \to B\rtimes \G\). Either the maximal or the reduced crossed-product 
can be used; thus there is also a `reduced' descent map 
\[ j_r \colon \KK^\G_*(A,B) \to \KK_*(A\rtimes_r \G, B\rtimes_r \G)\]
in which the reduced is used.

Descent \(j\) (respectively reduced descent \(j_r\)) makes 
the abelian group \(\KK(A\rtimes \G, B\rtimes \G)\) 
(respectively \(\KK(A\rtimes_r \G, B\rtimes_r \G)\) 
a left module over the ring \(\KK^\G_*(\C, \C)\), and likewise a 
right module, using the structure of \(\KK^\G(A, B)\) as a module over 
\(\KK^\G(\C, \C)\), for any \(\G\)-C*-algebras \(A,B\). 

The \(\gamma\)-element is defined contingent on the existence of a proper \(\G\)-C*-algebra \(P\) and 
classes 
\(\eta \in \KK^\G(\C, P)\) and \(D\in \KK^\G(\C, P)\) 
such that \(D\otimes_\C \eta = 1_P \in \KK^\G(P, P)\), as the
idempotent \(\eta\otimes_P D \in \KK^\G(\C, \C)\). 
The existence of a \(\gamma\)-element is not guaranteed for arbitrary discrete groups, but a group can have at 
most one \(\gamma\)-element, as one can argue without difficulty. The existence issue 
involves the existence of \(\eta\), called the \emph{dual-Dirac morphism}: it
 can be shown (see \cite{Meyer-Nest} and \cite{Emerson:dd}) that for any \(\G\), there exist
 proper \(P\) and a morphism \(D\in \KK^\G(P, \C)\) (the \emph{Dirac morphism}) such that existence of \(\eta\) 
 is equivalent to a coarse geometric condition on the group, namely, that the `coarse co-assembly map' for \(\G\) is an isomorphism (the coarse co-assembly map is described in \cite{Emerson:dd}). The coarse co-assembly map is, however, an isomorphism for all hyperbolic groups, and more generally, for groups which uniformly embed in a Hilbert space, so all such groups have \(\gamma\)-elements. The first explicit construction of them in the case of hyperbolic groups 
is due to Kasparov and Skandalis \cite{KS}. 

It is \emph{not} true that \(\gamma = 1\in \KK^\G_0(\C, \C)\) for general hyperbolic groups, \(1\) being the  class \(1:= [\epsilon] 
\in \KK_0^\G(\C, \C)\) of the trivial representation \(\epsilon \colon C^*\G \to \C\). An argument of Skandalis \cite{Ska} even gives examples where 
\(j_r (\gamma) \not=1_{C^*_r\G}\in \KK_0(C^*_r\G, C^*_r\G)\). 

For cocompact lattices in 
\(\mathrm{SO}(n,1)\) or \(\mathrm{SU}(n,1)\), or free groups, \(\gamma=1\) is true due to results of Kasparov \cite{Kas1}, and Kasparov and Julg \cite{JK}. These groups are also known to be a-T-menable, so \(\gamma = 1\) follows from the Higson - Kasparov theorem (see \cite{HK}) as well.

For our purposes, we are mostly concerned about 
whether \(\gamma\) \emph{acts} as the identity on various \(\KK\)-groups, especially \(\K^*(C^*_r\G)\). When \(\G\) is hyperbolic, recent work of Lafforgue and others \cite{Laf, MY} shows that \(\gamma\) \emph{does} act as the identity on the \(\K\)-\emph{theory} of \(C^*_r\G\), but nothing seems to be known at present about the case of \(\K\)-homology.

The point of the \(\gamma\)-part, is that it is the `topologically accessible' part of the \(\K\)-homology, in the 
sense of the following theorem which 
is essentially due to Kasparov. See Theorem 
23 of \cite{Emerson:dd} and Theorem 7.1 of 
\cite{Meyer-Nest}.

 \begin{lem}[Kasparov]\label{lembc} Let \(\G\) be a discrete group with a 
 \(\gamma\)-element and 
 \(\EG\) its classifying space for proper actions. Then the
  canonical inflation map of \cite{Kas}
 \[ p_{\EG}^*\colon \KK^\G_*(A,B) \to \RKK^\G_*(\EG; A,B)\]
is an isomorphism from the \(\gamma\)-part of \(\KK^\G_*(A,B)\) onto 
its target. 
\end{lem}

Here
\(\RKK^\G(\EG; A , B)\) is the \(\G\)-equivariant representable 
\(\K\)-theory of \(\EG\). If \(A = B = \C\) and \(G\backslash \EG\) is compact, then 
it agrees with the ordinary K-theory of 
\(C_0(\EG)\rtimes \G\), and if in addition \(G\) is torsion-free, then it is 
isomorphic to \(\K^*(\G\backslash \EG)\) -- see \cite{Emerson:rkk} for 
information on equivariant representable \(\K\)-theory.

Finally, we remind the reader that 
since a Gromov hyperbolic group acts amenably on its boundary, 
\(\gamma\) acts as the identity on \(\KK^\G_*(C(\bd \G)\otimes A, B)\) for 
any \(A,B\). (The Dirac dual-Dirac method gives a \(\KK^\G\)-equivalence 
between \(C(\bd \G)\) and a proper \(\G\)-C*-algebra, while \(\gamma\) acts 
as the identity on any \(\KK^\G(P, B)\)-group when \(P\) is proper, by 
properties of \(\gamma\) -- see \cite{Meyer-Nest}.)

\begin{cor}
\label{lemhkt}
For any \(\G\)-C*-algebras \(A,B\), 
\begin{align*} 
\KK^\G_*(C(\bd \G)\otimes A, B) &\cong \RKK^\G_*(\EG; C(\bd \G)\otimes A,B)\\
\KK^\G_*(C(\overline{\G})\otimes A, B) &\cong \RKK^\G_*(\EG; C(\overline{\G})\otimes A,B)
\end{align*}
by the inflation map $p_{\EG}^*$. \end{cor}

\subsection{The \(\gamma\)-element regarded as a K-homology class for \(C^*_r\G\).}\label{bloop}

%Since for any discrete group \(\G\), \(\KK^\G(\C, \C) \cong \K^*(C^*\G)\), 
%the \(\gamma\)-element for \(\G\) is a K-homology class for \(C^*\G\). 
%As such, it has ordinary Fredholm index \(1\); that is, the inclusion 
%\(\C \to C^*\G\) of the unit in the C*-algebra \(C^*\G\) pulls \(\gamma\) back to 
%the element \(1\in \KK(\C, \C)\).  

Let \(\lambda \colon C^*\G \to C^*_r\G\) be the projection from the maximal group C*-algebra to the reduced group C*-algebra, and let us call any element \(\gamma_r \in \K^0(C^*_r\G)\) such that 
\begin{align*}
\lambda^*(\gamma_r)=\gamma \in \K^0(C^*\G) \cong \KK^\G_0(\C, \C)
\end{align*}
a \emph{reduced \(\gamma\)-element} for \(\G\).

\begin{prop}
\label{indexoneelement}
The map 
\(\lambda^*\colon \K^*(C^*_r\G) \to \K^*(C^*\G) \cong \KK^\G_*(\C, \C)\) induces an isomorphism between the \(\gamma\)-parts of these two rings. In particular, if \(\G\) has a \(\gamma\)-element, then it has a reduced \(\gamma\)-element. 
\end{prop}

The (or any such) element \(\gamma_r\) as in the Proposition will play a role 
in the `Gysin sequence' developed in the next subsection.

\begin{proof}
We produce a map \(\gamma \K^*(C^*\G) \to \gamma \K^*(C^*_r\G)\) inverting \(\lambda^*\) as 
follows. We first recall that the standard identification  \(\KK^\G_*(A, \C) \to \KK_*(A\rtimes \G, \C)\), 
coming from the fact that the groups have the same cycles when \(\G\) is discrete, 
can be expressed in terms the 
`descent' construction and the trivial representation of the group in the following way: 
it is the composition 
of the descent map 
\(j\colon \KK^\G_*(A, \C) \to \KK_*(A\rtimes \G, C^*\G)\), and  
\(\epsilon_*\colon  \KK_*(A\rtimes \G, C^*\G)\to \KK_*(A\rtimes \G, \C)\).
In particular, taking \(\gamma \in \KK^\G_0(\C, \C)\), its image under the 
isomorphism with \(\KK_0(C^*\G, \C) = \K^0(C^*\G)\) is 
\(j(\gamma) \otimes_{C^*\G} [\epsilon]\).
With these formalities aside, we next factor the \(\gamma\)-element, or rather, its 
image in \(\KK_0(C^*\G, \C)\) as follows. Let \(\eta\in 
\KK^\G(\C, P)\) be the dual-Dirac morphism, and let \(D\in \KK^\G(P, \C)\) be the Dirac morphism for \(\G\). Then
 \(j (\eta)\otimes_{P\rtimes \G} j(D) \otimes_{C^*\G} [\epsilon] \) factors the image of \(\gamma\) in \(\KK_0(C^*\G, \C)\). This is because \(\gamma = \eta\otimes_P D\), and the naturality of the descent map \(j\). 
 
 More generally, any \(a\in \gamma \KK^\G_*(\C, \C)\), interpreted as an element of 
 \(\K^*(C^*\G)\), can be thus factored as 
 \begin{multline}
 \label{1}
  j (a) \otimes_{C^*\G} [\epsilon] = j (a\otimes_\C \gamma) \otimes_{C^*\G} [\epsilon]
  \\ = j (a\otimes_\C \eta \otimes_P D) \otimes_{C^*\G} [\epsilon]
  = j (a)\otimes_{C^*\G}  j( \eta) \otimes_{P\rtimes \G}  j(D) \otimes_{C^*\G} [\epsilon]
 \end{multline}
where the first equality is due to \(\gamma \otimes_\C a = a\) for 
\(a\) in the \(\gamma\)-part, the third by the naturality of the descent map.

 Now to obtain an element \(a'\) such that \(\lambda^*(a') = a\), consider the element  
 \begin{align*} b' = j_r (a) \otimes_{C^*_r\G} j_r( \eta) \in \KK_*(C^*_r\G, P\rtimes_r \G)
 \end{align*}
defined using the reduced descent map. Now \(P\) being proper implies \(P\rtimes_r\G \cong 
P\rtimes \G\). Applying this isomorphism to \(b'\) gives  a class \(b\in \KK_*(C_r^*\G, P\rtimes\G)\). Then the required element \(a'\) such that \(\lambda^*(a') = a\) is 
\begin{align*} a':= b\otimes_{P\rtimes \G} j(D) \otimes_{P\rtimes \G} [\epsilon].
\end{align*}
\end{proof}

\begin{rem}
\label{rem:differentgammaparts}
In particular, Kasparov's Theorem (Lemma \ref{lembc}) can be alternately phrased 
in terms of the \(\K\)-homology of the \emph{reduced} C*-algebra: the 
\(\gamma\)-part of the \(\K\)-homology of \(C^*_r\G\) is isomorphic to the 
topological group \(\RKK^\G(\EG; \C, \C)\) (by the composition of \(\lambda^*\) and 
the inflation map.) 
\end{rem}

\subsection{The Gysin sequence}
Let \(\G\) be a hyperbolic group, \(\bd \G\) its Gromov boundary, \emph{etc}. 
Let \(i_\G \colon \C \to C(\bd \G)\) be the natural inclusion of \(\C\) as constant 
functions on \(\bd \G\), defining a morphism \([i_\G]\in \KK^\G_0(\C, C(\bd \G))\) 
and then, by reduced descent, a morphism 
\([i]:= j_r([i_\G]) \in \KK_0(C^*_r\G, \hi)\), which is nothing but the Kasparov morphism 
determined by the C*-algebra injection \(i \colon C^*_r\G\to \hi\) of the reduced C*-algebra in the reduced crossed-product. 

 Then composition with \([i]\) induces 
a map \(i^*\colon \K^*(\hi) \to \K^*(C^*_r\G)\). The aim of this section is to 
compute this map. We first observe that the range of this 
map is contained in the \(\gamma\)-part of \(\K^*(C^*_r\G)\). More 
generally: 

\begin{lem}
\label{lem:contained_in_gamma_part}
Let \(A\) be any \(\G\)-C*-algebra and 
\(\alpha_\G  \colon A \to C(\bd \G)\) a \(\G\)-equivariant *-homomorphism. 
Let \(\alpha \colon A\rtimes_r \G \to \hi\) be the induced 
*-homomorphism. Then the range of the induced map \(\alpha^*\colon \K^*\bigl(\hi \bigr) \to 
\K^*(A\rtimes_r \G)\) is contained in the \(\gamma\)-part of 
\(\K^*(A\rtimes_r\G)\). 
\end{lem}

\begin{proof}
Since \(\gamma\) acts as the identity on \(\K^*(\hi)\)  and 
\(\alpha_\G^*\colon \KK^\G_*(C(\bd \G), \C) \to \KK^\G_*(A, \C)\)
 is a 
\(\KK^\G_*(\C, \C)\)-module map,  for any \(x\in \KK^\G_*(C(\bd \G), \C)\) it holds that 
\(i^*(x) = i^*(\gamma x) = \gamma i^*(x) \in \gamma \KK^\G_*(A, \C)\).
The result follows. 

\end{proof}

Let \(X\) be a Rips complex for \(\G\) which models \(\EG\) (see \cite{Mein}). Let 
\begin{itemize}[leftmargin=20pt, itemsep=3pt]
\item \(r\colon C(\overline{X})\to C(\bd X) \cong C(\bd \G)\) be the 
\(\G\)-equivariant map of restriction to the boundary, 
\item \(u \colon \C \to C(\overline{X})\) be the inclusion as constant functions.
\end{itemize}
Both maps are \(\G\)-equivariant. By 
Lemma \ref{lem:contained_in_gamma_part}, the range of 
\(r^*\) is contained in the \(\gamma\)-part of \(\KK_*^\G(C(\overline{X}), \C)\). 

\begin{lem}
\label{happiness1}
The map 
\begin{align*}
u^*\colon  \KK^\G_*(C(\overline{X}), \C)\rightarrow \KK^\G_*(\C, \C)
\end{align*}
on \(\KK^\G\)-theory induced by 
composition with \(u\in \KK^\G_0(\C, C(\overline{X}))\),
restricts to an \emph{isomorphism}
between the \(\gamma\)-parts of the domain and co-domain. 
 Moreover, the composition 
\begin{align*}
 \KK^\G_*(C(\bd \G), \C) \xrightarrow{r^*} \gamma \KK^\G_*(C(\overline{X}), \C) 
 \xrightarrow{u^*} \gamma \KK^\G_*(\C, \C)  \end{align*}equals \(i_\G^*\). 
 \end{lem}

\begin{proof}
Recalling that \(X = \EG\), Lemma \ref{lembc} says that the inflation map 
\begin{align*} 
p_X^*\colon \KK^\G_*(A, B) \to \RKK^\G_*(X; A, B)
\end{align*}
is an isomorphism from the \(\gamma\)-part of \(\KK^\G_*(A, B)\) to 
\(\RKK^\G_*(X; A, B)\). 
On the other hand,
 \(\overline{X}\) is \(H\)-equivariantly contractible for every finite subgroup 
\(H\) of \(\G\). In other worlds \(C(\overline{X})\) is \(H\)-equivariantly 
homotopy equivalent to \(\C\) for every finite \(H \subset \G\), equivalently, 
\(u\colon \C \to C(\overline{\G})\) regarded as an element of \(\KK^H_0(\C, C(\overline{X}))\) 
is invertible for every such \(H\). Hence by  
\cite{Meyer-Nest}, \(p_X^*(u)\) is invertible. So \(u^*\) is invertible 
between the \(\gamma\)-parts of \(\KK^\G_*(C(\overline{X}), \C)\) and 
\( \KK_*^\G(\C, \C)\). 

The last statement is left to the reader. 
\end{proof}

The basis of the arguments to follow is the 
\(\G\)-exact sequence 
\begin{equation} 
\label{equation: the_equivariant_sequence}
0 \to C_0(X) \to C(\overline{X}) \to C(\bd X)\cong C(\bd \G) \to 0
\end{equation}
of \(\G\)-C*-algebras. Let us make a few preliminary observations regarding excision in \(\KK\)-theory for this situation. 

Firstly, corresponding to the sequence \eqref{equation: the_equivariant_sequence} is a 
commutative diagram 
\begin{equation}
\xymatrix{ 0 \ar[r]& C_0(X)\rtimes_{\textup{max}}  \G \ar[r] \ar[d] & C(\overline{X})\rtimes _{\textup{max}}\G \ar[d]\ar[r] 
& C(\bd \G)\rtimes_{\textup{max}} \G \ar[r] \ar[d] &  0 \\  
0\ar[r] & C_0(X)\rtimes_r \G \ar[r]  & C(\overline{X})\rtimes_r \G \ar[r] 
&C(\bd \G)\rtimes_r \G \ar[r] & 0 \\  }
\end{equation}
of crossed-products, with exact rows. The vertical maps are 
the natural ones, from the maximal crossed-products to the reduced. 
The top row is exact because we are using the maximal crossed-product; 
the lower row is exact by exactness of \(\G\). As the first and 
third vertical maps are isomorphisms, so is the middle one. 
This give a quick proof that 
\[ C(\overline{X})\rtimes_{\textup{max}} \G \cong C(\overline{X})\rtimes_r \G,\]
which in turn implies that 
\[ \KK^\G_*(C(\overline{X}), \C) \cong \K^*\bigl( C(\overline{X})\rtimes_r \G).\]

Next, by nuclearity of \(\hi\), the exact sequence 
\[ 0 \to C_0(X)\rtimes \G \to C(\overline{X})\rtimes \G \to \hi \to 0 \]
in which all crossed-products are reduced, induces a long exact sequence of \(\K\)-homology groups
\begin{align}
\label{esk1}
\cdots \leftarrow \K^*(C_0(X)\rtimes \G) \leftarrow \K^*(C(\overline{X})\rtimes \G) 
\leftarrow \K^*(\hi) \leftarrow \cdots
\end{align}
Since all the \(\K\)-homology groups in this sequence are 
isomorphic to their equivariant counterparts, we can view this, 
and prefer to do so, as an exact sequence of 
equivariant \(\K\)-homology groups
\begin{align}
\label{esk1}
\cdots \leftarrow \KK^\G_*(C_0(X), \C) \leftarrow \KK^\G_*(C(\overline{X}), \C) 
\leftarrow \KK^\G_*(C(\bd X), \C) \leftarrow \cdots
\end{align}

The map \(\KK_*^\G( C(\overline{X}), \C) \to \KK^\G_*(C_0(X), \C)\) in this sequence 
will be denoted \(\varphi^!\): it is the map on K-homology induced by the equivariant $*$-homomorphism 
\(C_0(X) \to C(\overline{X})\), while the map \(\KK^\G_*(C(\bd \G), \C) \to \KK^\G_*(C(\overline{X}) ,\C)\)
is the map \(r^*\) on equivariant \(\K\)-homology induced by the restriction homomorphism \(r\colon C(\overline{X}) \to C(\bd X)\).

Now \(\gamma\) is an idempotent and all maps in this 
exact sequence are \(\KK^\G_0(\C, \C)\)-module maps so commute with 
\(\gamma\). It follows that, taking \(\gamma\)-parts, gives an exact 
sequence 
\begin{align}
\label{esk2}
\cdots \leftarrow \gamma \KK^\G_*(C_0(X), \C) \xleftarrow{\varphi^!}\gamma \KK^\G_*(C(\overline{X}), \C) 
\xleftarrow{r^*}\gamma \KK^\G_*(C(\bd X), \C) \leftarrow \cdots
\end{align}
Combining this sequence with the sequence \eqref{esk1}, which 
maps to it, and 
applying the Five Lemma gives that actually \(\gamma\) acts as the 
identity on all the groups in \eqref{esk1}. So \eqref{esk2}
can be used in place of \eqref{esk1}, as they are exactly the 
same sequence.

By Lemma \ref{happiness1} we can replace the middle term 
\(\gamma \KK^\G_*(C(\overline{X}), \C)\) in \eqref{esk2} by \(\gamma \KK^\G(\C, \C)\). With this 
replacement, the map \(r^*\) is replaced by (the map induced on \(\gamma\)-parts by) 
 \(i_\G^*\) since \(r\circ u = i_\Gamma\). 
Exactly as in \cite{Emerson:Euler} one computes that the map 
\(\varphi^!\) becomes 
the map induced on \(\gamma\)-parts by the map
\(\KK^\G_*(\C, \C) \to \KK^\G_*(C_0(X), \C)\) of external product with 
the \emph{Euler class} 
defined to be  
\begin{align*}
\Eul_\G := (p_X^*)^{-1}(\Delta_X)\in \KK^\G_0(C_0(X), \C)
\end{align*}
where \(\Delta_X\in \RKK^\G_0(X; C_0(X), \C)\) is the morphism induced by the 
diagonal embedding \(X\to X\times X\). (See \S 3 of \cite{Emerson:Euler}).

Now if \(\G\) is torsion-free then, since \(\G\backslash X\) is compact and models \(B\G\), we have 
\(\KK^\G(C_0(X), \C) \cong \K_0(\G\backslash X) \cong \K_0(B\G)\), and 
under this identification, the Euler class for \(\G\) is just the ordinary Euler characteristic of \(\G\) (an integer, equal to the Euler characteristic 
of \(B\G\)) multiplied by the class of a point in K-homology (see \cite{Emerson:Euler}). So we can insert this into the sequence 
\eqref{esk1} in the torsion-free case. Finally, using the fact that \(\gamma\) acts as the identity on both 
\(\KK^\G_*(C_0(X), \C)\) and on \(\KK^\G_*(C(\bd X), \C)\),  we obtain the following.

\begin{thm}[Gysin sequence for K-homology]\label{gysin}
Let \(\G\) be a torsion-free hyperbolic group. Then there is an exact sequence 
\begin{multline*}
\label{gysin1}
0 \to \K_1(B\G)\rightarrow \K^0(\hi) \xrightarrow{i^*} \gamma\KK^\G_0(\C, \C) \xrightarrow{\Eul} \K_0(B\G)\\ \to \K^1(\hi)\xrightarrow{i^*} \gamma\KK^\G_1(\C, \C) \to 0
\end{multline*}
where \(i^*\colon \K^*(\hi) \to \K^*(C^*_r\G)\) is the map induced by the 
inclusion \(i\colon C^*_r\G\to \hi\), and where \(\Eul\) is the map $\Eul (a) = \chi (\G)\: \ind(a)\:[\pnt]\in \K_0(B\G)$, with \(\ind\) the ordinary Fredholm index map \(\KK^\G(\C, \C) \to \Z\), and \([\pnt]\) is the class in \(\K\)-homology of a point in \(B\G\). 
\end{thm}

\begin{cor}\label{applied gysin}
The restriction homomorphism \(i^*\colon \K^*(\hi) \to \gamma \K^*(C^*_r\G)\) is a surjection in dimension \(*=1\), and a surjection in both dimensions if \(\chi (\G) = 0\). When \(\chi (\G) \not= 0\), let \(\gamma_r\in \gamma \K^0(C^*_r\G)\) be a reduced \(\gamma\)-element. Then for each \(a\in \gamma \K^0(C^*_r\G)\) there exists \(b\in \K^0(\hi)\) such that 
 \begin{align*}
 a = \ind (a) \gamma_r+  i^*(b).
 \end{align*}
\end{cor}

\begin{proof}
The statement regarding \(*=1\) and the one when the Euler characteristic is zero 
are both obvious from the Gysin sequence. For the second statement, let 
\(a\in \gamma \K^0(C^*_r\G)\), then since \(\ind (\gamma_r) =1\), \(a - \ind (a)\gamma_r\) has 
index zero and hence is in the kernel of the map \(\Eul\). Hence it is in the range of 
\(i^*\), by the Gysin sequence. Thus 
\(a = \ind (a) \gamma_r + i^*(b)\) for \(b\in \mathrm{ran}(i^*)\) as claimed. 
\end{proof}

The results of this section show that,
 up to a cyclic summand, the \(\K\)-homology of the reduced 
C*-algebra of \(\G\) comes entirely from restricting  \(\G\)-equivariant 
\(\K\)-homology classes from the boundary.

%%%%%%%%%%%%%%%%%%%%%%%%%%%
%%%%%%%%%%%%%%%%%%%%%%%%%%%
\section{Uniformly summable K-cycles over the reduced group C*-algebra}
Let $\G$ be a regular and torsion-free hyperbolic group. Recall that every class in $\K^1(\hiX)$ can be represented by an odd Fredholm module of the form
\begin{align*}
\big(\lambda_\mu, P_{\ell^2\G}\lambdaop_\mu(e)P_{\ell^2\G}\big)
\end{align*}
for some projection $e\in \hiX$, and every class in $\K^1(\hiX)$ can be represented by a balanced even Fredholm module of the form
\begin{align*}
\big(\lambda_\mu, P_{\ell^2\G}\lambdaop_\mu(u)P_{\ell^2\G}+(1-P_{\ell^2\G})\big)
\end{align*}
for some unitary $u\in \hiX$. At the level of cycles, the map \(i^*\) on \(\K\)-homology induced by \(i\colon 
C^*_r\G \to \hi\) merely restricts the representation of \(\hi\) to the subalgebra \(C^*_r\G\). Thus we restrict 
the representation \(\lambda_\mu \) to \(C_r^*\G\). Then, as each \(\lambda_\mu(g)\) commutes with \(P_{\ell^2\G}\), we can remove the degenerate summand $(1-P_{\ell^2\G})\cdot \ell^2(\G, L^2(\bd X, \mu))$. Note that the restriction of \(\lambda_\mu\) to 
the remaining summand $P_{\ell^2\G}\cdot \ell^2(\G, L^2(\bd X, \mu))= 
\ell^2\G$ is the regular representation $\lambda$. Thus, over \(C^*_r\G\) the above Fredholm modules take the form
\begin{align*}
\Phi(a):=\big(\lambda, P_{\ell^2\G}\lambdaop_\mu(a)P_{\ell^2\G}\big)
\end{align*}
where $a$ is a projection or a unitary in $\hiX$. If $a$ is a projection or a unitary in $\mathcal{A}$, where $\mathcal{A}$ is as in Theorem~\ref{enough are summable}, then $\Phi(a)$ is $D_\e^>$-summable over the group algebra $\C\G$.

\begin{lem}
\label{lemrangeofi}
Assume that $\G$ is regular and torsion-free, and let \(\mathcal{A}\) be the smooth subalgebra of Theorem \ref{enough are summable}. Then every class in the image of the restriction map \(i^*\colon \K^*(\hiX) \to \K^*(C^*_r\G)\) is represented by a Fredholm module of the form $[\Phi(a)]$ for some projection, respectively unitary $a\in \mathcal{A}$. In particular, every class in $i^*\K^*(C^*_r\G)$ has a representative which is $D_\e^>$-summable over $\C\G$. 
\end{lem}

Combining Lemma~\ref{lemrangeofi} and Corollary~\ref{applied gysin}, we obtain:

\begin{thm}
\label{corbigone}
Assume that $\G$ is regular and torsion-free, and let \(\mathcal{A}\) be the smooth subalgebra of Theorem \ref{enough are summable}. Then the following hold.

\begin{itemize}[leftmargin=20pt, itemsep=3pt]
\item Every class in $\gamma \K^1(C^*_r\G)$ is of the form $[\Phi(e)]$ for some projection $e\in \mathcal{A}$. In particular, every class in $\gamma \K^1(C^*_r\G)$ is represented by a Fredholm module which is $D_\e^>$-summable over $\C\G$.

\item If $\chi (\G) = 0$, then every class in $\gamma \K^0(C^*_r\G)$ is of the form $[\Phi(u)]$ for some unitary $u\in \mathcal{A}$. In particular, every class in $\gamma \K^0(C^*_r\G)$ is represented by a Fredholm module which is $D_\e^>$-summable over $\C\G$.

\item If \(\chi (\G)\not=0\), and \(\gamma_r\) is a reduced \(\gamma\)-element, then every class in $\gamma \K^0(C^*_r\G)$ is of the form $k\gamma_r+[\Phi(u)]$ for some integer $k$ and some unitary $u\in \mathcal{A}$. In particular, if $\gamma_r$ is represented by a Fredholm module which is $p(\gamma_r)$-summable over $\C\G$, then every class in $\gamma \K^0(C^*_r\G)$ is represented by a Fredholm module which is $\max\{p(\gamma_r), D_\e^>\}$-summable over $\C\G$. 
\end{itemize}
\end{thm}

We now specialize Theorem~\ref{corbigone} to four families of a-T-menable groups; note that \(\gamma = 1\) for a-T-menable groups by \cite{HK}, so \(\gamma \K^*(C^*_r\G) = \K^*(C^*_r\G)\).

\subsection{Free groups} Let $\G$ be a finitely generated free group of rank at least $2$. Given any \(p>2\), every class in 
$i^*\K^*(\hi)$ has a $p$-summable representative over $\C\G$. On the other hand, the Julg - Valette model for the $\gamma$-element \cite{JV} is $1$-summable over $\C\G$, hence the same holds true for the reduced $\gamma$-element $\gamma_r$. We conclude that $C^*_r\G$ has uniformly $p$-summable K-homology over $\C \G$ for every $p>2$.

\subsection{Real uniform lattices} Let \(\G\) be a torsion-free uniform lattice in \(\SO(n,1)\). Then classes in $i^*\K^*(\hi)$ are $(n-1)^+$-summable over $\C\G$ when $n\geq 4$, respectively $p$-summable over $\C\G$ for every $p>2$, when $n=2,3$.

If $n$ is odd, then $\chi(\G)=0$ so $i^*\K^*(\hi)$ covers in fact all the K-homology of \(C^*_r\G\). On the other hand, but still in this odd case, Kasparov shows in \cite{Kas1} that the \(\gamma\)-element for \(\SO(n,1)\) is represented by a Fredholm module in which the operator \(F\) is the phase of a degree $1$ elliptic operator on the sphere \(S^{n-1}\). Moreover, the unitary action of the 
group $\G$ commutes with $F$ modulo pseudodifferential operators of order $-1$ because the action is conformal (and so the operators $F$ and $gFg^{-1}$ have the same symbol). Hence the commutators \([g, F]\) have singular values satisfying $s_k\asymp k^{-1/(n-1)}$, that is, Kasparov's Fredholm module is $(n-1)^+$-summable over $\C\G$. 
Finally, Kasparov's equivariant Fredholm module supplies a reduced \(\gamma\)-element, that is, the group representations involved are weakly contained in the 
regular representation, since they factor through representations of
 \(C(S^{2n-1})\rtimes \G\), and \(\G\) acts amenably on \(S^{2n-1}\).

If \(n\) is even, we can make small adjustments to this argument. The pull-back of 
the \(\gamma\)-element in \(\KK^{\SO(n+1, 1)}_0(\C, \C)\) to an element of \(\KK^\G_0(\C, \C)\) 
under the inclusion \(\G \subset \SO(n,1)\subset \SO(n+1, 1)\) of \(\G\) as a closed subgroup of \(\SO(n+1, 1)\), is the 
\(\gamma\)-element for \(\G\). By Kasparov's constructions described in the previous paragraph, 
we have therefore a description of the \(\gamma\)-element for \(\G\) as a Fredholm module in which the 
group representations are weakly contained in the regular representation, and which is \(n^+\)-summable 
over \(\C \G\). We obtain therefore a reduced \(\gamma\)-element with the same properties.

We conclude that the K-homology of \(C^*_r\G\) is uniformly $n^+$-summable over \(\C \G\) when
 $n\geq 3$, respectively $p$-summable over $\C\G$ for every $p>2$, when $n=2$.

\subsection{Complex uniform lattices} Let \(\G\) be a torsion-free uniform lattice in \(\SU(n,1)\). Then classes in $i^*\K^*(\hi)$ are $(2n)^+$-summable over $\C\G$. A model for the \(\gamma\)-element of \(\SU(n,1)\) has been given by Julg and Kasparov in \cite{JK}. In this case, the method involves construction of an appropriate hypoelliptic operator on the contact manifold \(S^{2n-1}\) (the contact structure is \(\SU(n,1)\)-invariant.)
 Inspection of the article \cite{JK} reveals that the  
relevant commutators \([g,F]\) are pseudodifferential operators in the class \(\Psi^{-1}_H(S^{2n-1})\), and it is well-known that the singular values in this case satisfy \(s_k\asymp k^{-1/(2n)}\). We conclude that the K-homology of \(C^*_r\G\) is uniformly \((2n)^+\)-summable over \(\C \G\).

\subsection{Small-cancellation groups} Let $\G$ be a torsion-free group given by a finite presentation $\la S\: | \: \mathcal{R}\ra$ satisfying the $C'(1/6)$ small-cancellation condition. As a geometric model for $\G$ we take the Cayley graph with respect to $S$, denoted $\G(S)$. 

Firstly, let us point out an explicit estimate for the visual dimension of the boundary of $\G(S)$. Combining Fact~\ref{small visual range} and Fact~\ref{Ahlfors}, we get the coarse estimate $\visdim \bd \G(S)\leq 5\delta\: e_{\G(S)}\leq 5\delta \log(2|S|-1)$, where $\delta$ is the hyperbolicity constant of $\G(S)$. In a $C'(1/6)$ situation, it is possible to give a combinatorial estimate for $\delta$, namely, $\delta\leq 3 \max\{|r|: r\in \mathcal{R}\}$ by \cite[Appendix, Thm.36]{GdH}. We thus get the explicit, though far from optimal, bound
\begin{align*}
\visdim \bd \G(S)\leq 15\log(2|S|-1)\max\{|r|: r\in \mathcal{R}\}=: \kappa(S | \mathcal{R}).
\end{align*}

Secondly, let us argue that $\G$ is regular, in the sense of Definition~\ref{defn-regular}. As $\G$ is torsion-free, the $2$-complex defined by the $C'(1/6)$ presentation $\la S\: | \: \mathcal{R}\ra$ is aspherical. Hence $\chi(\G)=1-|S|+|\mathcal{R}|$, and $\G$ has cohomological dimension at most $2$. If $\mathrm{cd}\: \G=1$ then, by a well-known theorem of Stallings, $\G$ is a free group, and this is a case we have already discussed. So let us assume that $\mathrm{cd}\: \G=2$. A theorem of Bestvina and Mess (see, e.g., \cite[Thm.6.5]{KB}) implies that $\bd \G(S)$ has topological dimension $1$. We then have the following fact, due to Misha Kapovich (personal communication):

\begin{lem}\label{1-dim is regular}
If $\G$ is a hyperbolic group whose boundary has topological dimension $1$, then $\G$ is regular.
\end{lem}

\begin{proof}
A result of Bonk and Kleiner \cite{BK} says that the boundary of a hyperbolic group contains a quasi-circle provided that the group is not virtually free. In particular, $\bd\G$ contains a topological circle. The proof is completed by the following general claim: if $Z$ is a $d$-dimensional compact space containing a $d$-dimensional topological sphere $S^d$, then $Z$ admits a continuous self-map without fixed points. 

To prove the claim, recall the following alternative definition of topological dimension: a compact space $X$ has dimension at most $n$ if and only if every continuous map $X_0\to S^n$, defined on a compact subset $X_0\subseteq X$, can be continuously extended to the entire $X$. Applying this fact to the space $Z$ and the identity map $S^d\to S^d$, we obtain a retraction $\rho: Z\to S^d$. The composition $\tau \rho$, where $\tau: S^d\to S^d$ is the antipodal involution, is clearly fixed-point free.
\end{proof}

Thus $\G$ meets the conditions of Theorem~\ref{bigone applied}. It follows that the odd K-homology $\K^1(C^*_r\G)$ is uniformly $p$-summable over $\C\G$ for every $p > \kappa(S | \mathcal{R})$, and that the same is true for the even K-homology $\K^0(C^*_r\G)$ provided that $\chi(\G)=0$, i.e., $\G$ has deficiency $|S|-|\mathcal{R}|=1$.

%%%%%%%%%%%%%%%%%%%%%%%%%%%%%%%%%%%%%%%%%%%%%%%%%%%%%%%%%%%%%%%%%%%

\end{document}